\documentclass{article}

\author{Ofir Gorodetsky}
\title{Sums of two squares are strongly biased towards quadratic residues}
\date{}

\usepackage[margin=1in]{geometry}
\usepackage[all]{xy}
\usepackage{amssymb}
\usepackage{amsfonts}
\usepackage{amsthm}
\usepackage{amsmath}
\usepackage{hyperref}
\usepackage{mathtools}
\usepackage{url}
\usepackage{color}
\usepackage{multirow}
\usepackage{subfig}
\usepackage{booktabs}
\usepackage{float}
\usepackage{tikz}
\usetikzlibrary{calc,decorations.markings,positioning}

\mathtoolsset{showonlyrefs}

\newtheorem*{thm*}{Theorem}
\newtheorem{thm}{Theorem}[section]

\newtheorem{lem}[thm]{Lemma}  
\newtheorem{proposition}[thm]{Proposition}
\newtheorem{cor}[thm]{Corollary}
\theoremstyle{definition}

\newtheorem{remark}[thm]{Remark}

\newcommand{\CC}{\mathbb{C}}

\newcommand{\sosind}{\mathbf{1}_{S}}

\newcommand{\quadratic}{\chi_{-4}}

\numberwithin{equation}{section}

\newcommand{\Addresses}{{
		\bigskip
		\footnotesize
		
		\textsc{Mathematical Institute, University of Oxford, Oxford, OX2 6GG, UK}\par\nopagebreak
		\textit{E-mail address:} \texttt{ofir.goro@gmail.com}
}}

\begin{document}
	\maketitle
	\begin{abstract}
	Chebyshev famously observed empirically that more often than not, there are more primes of the form $3 \bmod 4$ up to $x$ than of the form $1 \bmod 4$. This was confirmed theoretically much later by Rubinstein and Sarnak in a logarithmic density sense. Our understanding of this is conditional on the generalized Riemann hypothesis as well as on the linear independence of the zeros of $L$-functions.
	
	We investigate similar questions for sums of two squares in arithmetic progressions. We find a significantly stronger bias than in primes, which happens for almost all integers in a \emph{natural density} sense. Because the bias is more pronounced, we do not need to assume linear independence of zeros, only a Chowla-type conjecture on nonvanishing of $L$-functions at $1/2$.	
	To illustrate, we have under GRH that the number of sums of two squares up to $x$ that are $1 \bmod 3$ is greater than those that are $2 \bmod 3$ 100\% of the time in natural density sense.
	\end{abstract}
	\section{Introduction}
	\subsection{Review of sums of two squares in arithmetic progressions}
	Let $S$ be the set of positive integers expressible as a sum of two perfect squares. We denote by $\sosind$ the indicator function of $S$. It is multiplicative and for a prime $p$ we have \begin{equation}\label{eq:characterization}
		\sosind(p^k)=0 \text{ if and only if }p\equiv 3 \bmod 4 \text{ and }2 \nmid k.
	\end{equation}
	Landau \cite{Landau1908} proved that \[ \# (S \cap [1,x]) \sim \frac{Kx}{\sqrt{\log x}}\]
	where $K=\prod_{p\equiv 3 \bmod 4}(1-p^{-2})^{-1/2}/\sqrt{2}\approx 0.764$ is the Landau-Ramanujan constant. See \cite[Lec.~IV]{Hardy1940} for Hardy's account of Ramanujan's unpublished work on this problem.
	Landau's method yields an asymptotic expansion in descending powers of $\log x$, which gives an error term $O_k(x/(\log x)^{k+1/2})$ for each $k\ge 1$.\footnote{A more complicated main term, leading to a significantly better error term (conjecturally $O_{\varepsilon}(x^{1/2+\varepsilon})$, but no better than that), is described e.g. in \cite[App.~B]{Gorodetsky2021} (compare \cite{Ramachandra1976}, \cite[p.~187]{Montgomery2007}, \cite[Thm.~2.1]{David2021}).}
	Prachar \cite{Prachar1953} proved  that sums of two squares are equidistributed in arithmetic progressions, in the following sense. If $(a,q)=1$ then
	\[ S(x;q,a) :=\#\{ n \in S: n \le x, \, n \equiv a \bmod q \} \sim  \frac{(4,q)}{(2,q)q}\prod_{\substack{p\mid q\\p \equiv 3 \bmod 4}} \left(1+\frac{1}{p}\right)\frac{Kx}{\sqrt{\log x}}\]
	as $x \to \infty$ as long as $a \equiv 1 \bmod (4,q)$; see Iwaniec's work on the half-dimensional sieve for results allowing $q$ to vary with $x$ \cite{Iwaniec1976}. The condition $a \equiv 1 \bmod (4,q)$ is necessary: otherwise $(a,q)=1$ and $a \not\equiv 1 \bmod (4,q)$ imply $a \equiv 3 \bmod 4$. However, $S$ is disjoint from $3 \bmod 4$.
	\subsection{Main theorem and corollary}
	Here we consider a Chebyshev's bias phenomenon for $S$. We ask, what can be said about the size of the set
	\begin{equation}\label{eq:bias}
		\{ n \le x: S(n;q,a)>S(n;q,b)\}
	\end{equation}
	for distinct $a,b \bmod q$ with $a\equiv b \equiv 1 \bmod (4,q)$ and $(a,q)=(b,q)=1$? These conditions guarantee that $S(n;q,a)\sim S(n;q,b) \to \infty$ as $n \to \infty$, so it is sensible to study \eqref{eq:bias}. We let $\quadratic$ be the unique nonprincipal Dirichlet character modulo $4$. Motivated by numerical evidence (based on $n \le 10^8$) showing $S(n;3,1)-S(n;3,2)$ and $S(n;5,1)-S(n;5,3)$ are positive much more frequently than not, we were led to discover and prove the following.
	\begin{thm}\label{thm:bias}
		Fix a positive integer $q $. Assume that the Generalized Riemann hypothesis (GRH) holds for the Dirichlet $L$-functions $L(s,\chi)$ and $L(s,\chi \quadratic)$ for all Dirichlet character $\chi$ modulo $q$.		
		Then, whenever $a,b$ satisfy $a\equiv b \equiv 1 \bmod (4,q)$, $(a,q)=(b,q)=1$ and \begin{equation}\label{eq:lincond} C_{q,a,b}:=\sum_{\substack{\chi \bmod q\\ \chi^2 = \chi_0}} (\chi(a)-\chi(b)) \left(1-\frac{\chi(2)}{\sqrt{2}}\right)^{-1/2}\sqrt{L\left(\frac{1}{2},\chi\right)L\left(\frac{1}{2},\chi \quadratic\right)}>0
		\end{equation}
		we have, as $x \to \infty$,
		\[ \# \{n \le x: S(n;q,a) > S(n;q,b)\} = x(1+o(1)). \]
	\end{thm}
Here (and later) $\chi_0$ is the principal character modulo $q$. Observe that $\chi(a)=\chi(b)=1$ for $\chi=\chi_0$ as well as for $\chi=\chi_0\quadratic$ (if $4 \mid q$), so these two characters may be omitted from the sum \eqref{eq:lincond}. As we explain in Remark~\ref{rem:nonneg} below, $L(1/2,\chi)L(1/2,\chi\quadratic)$ is nonnegative under the conditions of Theorem~\ref{thm:bias}; the squareroot we take in \eqref{eq:lincond} is the nonnegative one.
\begin{remark}\label{rem:nonneg}
Under GRH for (nonprincipal) real $\chi$, we have $L(1/2,\chi)\ge 0$ since otherwise there is a zero of  $L(s,\chi)$ in $(1/2,1)$ by the intermediate value theorem. Similarly, $L(1/2,\chi \quadratic) \ge 0$ if $\chi \neq \chi_0 \quadratic$. 
Conrey and Soundararajan proved, unconditionally, that for a positive proportion of quadratic characters $\chi$ we have $L(s,\chi)> 0$ on $(1/2,1)$ \cite{Conrey2002}, implying $L(1/2,\chi) \ge 0$. 
Chowla's Conjecture \cite{Chowla1965} states that $L(1/2,\chi) \neq 0$ for all real Dirichlet characters. It was studied extensively, see e.g. Soundararajan \cite{Sound2000}.
\end{remark}
	If $\chi(a)=1$ for all real characters modulo $q$ then $a$ is a quadratic residue modulo $q$, and vice versa.
	Let us specialize $a$ to be a quadratic residue modulo $q$ and $b$ to be a nonquadratic residue. We observe that (on GRH) 
	$(\chi(a)-\chi(b))\sqrt{L(1/2,\chi)L(1/2,\chi \quadratic)}$ is nonnegative for each real $\chi \bmod q$ that is not $\chi_0$ or $\chi_0 \quadratic$. Hence, under GRH and our assumptions on $a$ and $b$, a necessary and sufficient condition for \eqref{eq:lincond} to hold is that $L(1/2,\chi)L(1/2,\chi\quadratic) \neq 0$ for some real $\chi \bmod q$ with $\chi(b)=-1$. One way to guarantee this is to assume Chowla's Conjecture. We state this as the following corollary.
	\begin{cor}\label{cor:quad}
		Suppose that $GRH$ holds for $\chi$ and $\chi \quadratic$ as $\chi$ varies over all Dirichlet characters modulo $q$. 
		Let $a$ and $b$ be quadratic and nonquadratic residues modulo $q$, respectively, with $(a,q)=(b,q)=1$ and $a \equiv b \equiv 1 \bmod (4,q)$.	If $L(1/2,\chi)L(1/2,\chi\quadratic) \neq 0$ for some $\chi \bmod q$ with $\chi(b)=-1$ then  \begin{equation}\label{eq:sn ineq}S(n;q,a)>S(n;q,b)
		\end{equation}
	holds for a density-1 set of integers. In particular, if Chowla's Conjecture holds then \eqref{eq:sn ineq} holds for a density-1 set of integers.
	\end{cor}
It would be interesting to try and establish the positivity of \eqref{eq:lincond}, possibly in a statistical sense, without hypotheses like Chowla's Conjecture.
 
	For a given Dirichlet character $\chi$, one can computationally verify that $L(1/2,\chi)$ is nonzero, and in fact compute all zeros of $L(s,\chi)$ up to a certain height, see Rumely \cite{Rumley1993} which in particular shows $L(1/2,\chi) \neq 0$ for characters of conductor $\le 72$. Nowadays computing $L(1/2,\chi)$ is a one-line command in \emph{Mathematica}, and so the verification of \eqref{eq:lincond} is practical for fixed $q$, $a$ and $b$. 
	
	We expect the expression in \eqref{eq:lincond} to be nonzero as long as $\chi(a)\neq\chi(b)$ for some real character, or equivalently, if $a/b$ is nonquadratic residue modulo $q$. It is instructive to consider the following two possibilities for $a$ and $b$ separately:
	\begin{itemize}
		\item Suppose $a$, $b$ and $a/b$ are all nonquadratic residues, a situation that could occur only if the modulus $q$ is composite. Although the expression $C_{q,a,b}$ should be nonzero and give rise to a bias, it seems the sign is very difficult to predict. Interestingly, for primes, as we shall review below, there is no bias in this case.
		\item If exactly one of $a$ and $b$ is a quadratic residue then Corollary~\ref{cor:quad} tells us the direction of the bias (if it exists) is towards the quadratic residue. A sufficient condition for the bias to exist is Chowla's Conjecture.		
	\end{itemize}

	\subsection{Comparison with primes}
	Chebyshev's bias was originally studied in the case of primes, that is, replacing $S$ by the set of primes. Letting $\pi(x;q,a)$ be the numbers of primes up to $x$ lying in the arithmetic progression $a \bmod q$, Chebyshev famously observed in 1853 that $\pi(x;4,3)>\pi(x;4,1)$ happens more often than not \cite[pp.~697--698]{Tchebychef}.
	
	Littlewood \cite{Littlewood1914} showed that $\pi(x;4,3)-\pi(x;4,1)$ changes sign infinitely often. Knapowski and Tur\'{a}n \cite{Knapowski1962} conjectured that $\pi(x;4,3)>\pi(x;4,1)$ holds 100\% of the time in natural density sense. This was refuted, under GRH, by Kaczorowski \cite{Kaczorowski1992,Kaczorowski1995}, who showed (conditionally) that $\{ x:\pi(x;4,3)>\pi(x;4,1) \}$ does not have a natural density, and that its upper natural density is strictly less than 1.
	
	 Rubinstein and Sarnak \cite{Rubinstein1994} studied the set $\{ x : \pi(x;q,a)>\pi(x;q,b)\}$ where $a \not\equiv b \bmod q$ and $(a,q)=(b,q)=1$. 
	 They showed, under GRH and the Grand Simplicity Hypothesis (GSH) that this set has logarithmic density strictly between $0$ and $1$. Additionally, the logarithmic density is greater than $1/2$ if and only if $a$ is a nonquadratic residue and $b$ is a quadratic residue. In particular,  no bias is present at all if both $a$ and $b$ are nonquadratic residues, as opposed to the sums of two squares analogue.
	
	GSH asserts that the multiset of $\gamma \ge 0$ such that $L(1/2+i\gamma,\chi)=0$, for $\chi$ running over primitive Dirichlet characters, is linearly independent over $\mathbb{Q}$; here $\gamma$ are counted with multiplicity. It implies Chowla's Conjecture (since $0$ is linearly dependent) and that zeros of $L(s,\chi)$ are simple. As opposed to Chowla, it is very hard to gather evidence for GSH, even for individual $L$-functions. However, see Best and Trudgian for such evidence in the case of $\zeta$ \cite{Best2015}. In the literature, this hypothesis also goes under the name Linear independence (LI).
	\subsection{Strong biases}
	Chebyshev's bias was studied in various settings and for various sets, e.g. \cite{Ng2000,Moree2004,Fiorilli2014,Fiorilli20142,Devin2020,Bailleul2021,Devin20212}, in particular for products of a fixed number of primes \cite{Dummit2016,Ford2010,Meng2018,Devin2021}. 
	
	As far as we are aware, Theorem~\ref{thm:bias} is the first instance where a set of integers of arithmetic interest -- in this case sums of two squares -- is shown to exhibit a complete Chebyshev's bias, that is, a bias that holds for a natural density-$1$ set of integers:
	\[ \# \{ n \le x: M(n;q,a)>M(n;q,b)\} = x(1+o(1))\]
	where $M(n;q,a)$ counts elements up to $n$ in a set $M\subseteq \mathbb{N}$ that are congruent to $a$ modulo $q$.
	A key issue here is the natural density: Meng \cite{Meng2020} has a related work about a bias that holds for a logarithmic density-$1$ set of integers (see \S\ref{sec:martin conj}). Recently, Devin proposed a conjecture \cite[Conj.~1.2]{Devin20212} on a bias in logarithmic density $1$.
See Fiorilli \cite{Fiorilli20142,Fiorilli2014} for biases, in logarithmic density, that come arbitrarily close to $1$, and Fiorilli and Jouve \cite{fiorilli2020unconditional} for complete biases in `Frobenius sets' of primes (that generalize arithmetic progressions).
	
	We also mention a very strong bias was proved by Dummit, Granville and Kisilevsky \cite{Dummit2016} who take Chebyshev's observation to a different direction. They show that substantially more than a quarter of the odd integers of the form $pq$ up to $x$, with $p$, $q$ both prime, satisfy $p \equiv q \equiv 3 \bmod 4$.
	
	In the function field setting, complete biases were established in various special situations (e.g. for low degree moduli): \cite[Cor.~4.4]{Cha2008}, \cite[Thm.~1.5]{Cha2016}, \cite[Ex.~6]{Devin2021}. See also the work of Porritt \cite{Porritt2020} on which we elaborate in \S\ref{sec:generating}.
	\section{Origin of the bias, computational evidence and a variation}
\subsection{Review of the original bias}
Fix a modulus $q$. All the constants below might depend on $q$. We give an informal explanation for the origin of the bias. It is instructive to start with the case of primes. By orthogonality of characters,
\[ \pi(x;q,a)-\pi(x;q,b) = \frac{1}{\phi(q)} \sum_{\chi_0 \neq \chi \bmod q} \overline{(\chi(a)-\chi(b))}\sum_{n \le x} \mathbf{1}_{n \text{ is a prime}} \chi(n).\]
The generating function of primes was studied by Riemann \cite{Riemann1859}, who showed that
\[ \sum_{n \ge 1} \frac{ \mathbf{1}_{n \text{ is a prime}}}{n^s} = \sum_{k \ge 1} \frac{\mu(k)}{k} \log \zeta(sk)\]
for $\Re s > 1$. Here $\mu$ is the M\"obius function, $\zeta$ is the Riemann zeta function and the logarithm is chosen so that $\log \zeta(s)$ is real if $s$ is real and greater than $1$.
More generally, given a Dirichlet character $\chi$  we have
\[ \sum_{n \ge 1} \frac{\mathbf{1}_{n \text{ is a prime}}\chi(n) }{n^s} = \sum_{k \ge 1} \frac{\mu(k)}{k} \log L(sk,\chi^k).\]
We may also write this $L$-function identity in terms of arithmetic functions:
\begin{align}\label{eq:primes} \mathbf{1}_{n \text{ is a prime}}\chi(n)  &= \frac{\Lambda(n)\chi(n)}{\log n} +\sum_{k \ge 2} \frac{\mu(k)\Lambda(n^{1/k})\chi(n)}{\log n} \mathbf{1}_{n \text{ is a kth power}} \\
	&=\frac{\Lambda(n)\chi(n)}{\log n} - \frac{1}{2} \mathbf{1}_{n =p^2, \, p\text{ prime}} \chi^2(p) + \alpha(n)\chi(n),
\end{align}
where $\Lambda$ is the von Mangoldt function and  $\alpha$ is supported only on cubes and higher powers and its sum is negligible for all practical purposes.
Under GRH we can show that (see \cite[Lem.~2.1]{Rubinstein1994}) \[\sum_{n \le x} \frac{\chi(n)\Lambda(n)}{\log n} = \frac{\sum_{n \le x}\chi(n)\Lambda(n)}{\log x} + O\left( \frac{\sqrt{x}}{(\log x)^2} \right)\]
and (still under GRH) we can use the explicit formula to show that $\sum_{n \le x} \chi(n)\Lambda(n)$ is typically of order $\asymp \sqrt{x}$, in the sense that
\begin{equation}\label{eq:l2 primes} \frac{1}{X} \int_{X}^{2X} \left| \frac{\sum_{n \le x} \chi(n)\Lambda(n)}{\sqrt{x}}\right|^2 dx \ll 1,
\end{equation}
see \cite[Thm.~13.5]{Montgomery2007}. Under linear independence, one can show that the random variable
\[e^{-y/2}\sum_{n \le e^y} \chi(n)\Lambda(n)\]
has a limiting distribution with expected value $0$ (here $y$ is chosen uniformly at random between $0$ and $Y$, and $Y \to \infty$). The exponential change of variables leads to the appearance of logarithmic density. To summarize, $(\log x /\sqrt{x})\sum_{n \le x} \chi(n) \Lambda(n) /\log n$ (with $x=e^y$) has expectation $0$ and order of magnitude $\asymp 1$. The bias comes from the term $- \mathbf{1}_{n=p^2, \, p\text{ prime}}\chi^2(p)/2$. Indeed,
\[ -\frac{1}{2}\sum_{n \le x} \mathbf{1}_{n=p^2, \, p\text{ prime}} \chi^2(p)= -\frac{1}{2} \sum_{p \le \sqrt{x}} \chi^2(p).\]
If $\chi$ is a nonreal character, $\chi^2$ is nonprincipal and GRH guarantees this sum is $o(\sqrt{x}/\log x$). However, if $\chi$ is real, this sum is of the same order of magnitude as $\sum_{n \le x} \chi(n) \Lambda(n)/\log n$, namely it is
\[ -\frac{1}{2}\sum_{n \le x} \mathbf{1}_{n=p^2, \, p\text{ prime}} \chi^2(p) = -\frac{1}{2} (\pi(\sqrt{x})+O(1)) \sim -\frac{\sqrt{x}}{\log x}\]
using the Prime Number Theorem.

Rubinstein and Sarnak replaced $\sum_{n \le x} \chi(n) \Lambda(n)/\log n$ with $\sum_{n \le x} \chi(n) \Lambda(n)/\log x$ using partial summation. This is advantageous as we have a nice explicit formula for the sum of $\chi(n)\Lambda(n)$. However, one can work directly with $\chi(n) \Lambda(n)/\log n$, whose generating function is $\log L(s,\chi)$, and this was done by Meng \cite{Meng2018} in his work on Chebyshev's bias for products of $k$ primes. Meng's approach is more flexible because it works even when the generating function has singularities which are not poles. So while $-L'(s,\chi)/L(s,\chi)=\sum_{n \ge 1} \Lambda(n)\chi(n)/n^s$ is meromorphic with simple poles at zeros of $L(s,\chi)$, which leads to the explicit formula by using the Residue theorem, Meng's approach can deal with $-\log L(s,\chi)$ directly although it does not have poles, rather it has essential singularities.
Meng's work applies in particular to $k=1$ and $k=2$, thus generalizing Rubinstein and Sarnak as well as Ford and Sneed \cite{Ford2010}.
\subsection{The generating function of sums of two squares}\label{sec:generating}
Let us now return to sums of two squares. Let $a,b$ be residues modulo $q$ with $(a,q)=(b,q)=1$ and $a\equiv b \equiv 1 \bmod (4,q)$. Orthogonality of characters shows 
\begin{equation}\label{eq:sdiff}
 S(x;q,a)-S(x;q,b) = \frac{1}{\phi(q)} \sum_{\chi_0 \neq \chi \bmod q} \overline{(\chi(a)-\chi(b))}\sum_{n \le x} \sosind(n) \chi(n).
\end{equation}
We want to relate $\sum_{n \le x}\sosind(n) \chi(n)$ to $L$-functions and their zeros, and obtain an analogue of the explicit formula for primes. The generating function of sums of two squares was studied by Landau \cite{Landau1908}, who showed that for $\Re s > 1$,
\begin{equation}\label{eq:landform} \sum_{n \ge 1} \frac{ \sosind(n)}{n^s} = \sqrt{\zeta(s)L(s,\quadratic)} H(s)
\end{equation}
where $H$ has analytic continuation to $\Re s >1/2$. Here the square root is chosen so that $\sqrt{\zeta(s)}$ and $\sqrt{L(s,\quadratic)}$ are real and positive for $s$ real and greater than $1$. This representation of the generating function plays a crucial role in the study of the distribution of sums of two squares, see e.g. \cite{Gorodetsky2021}.

Later, Shanks \cite[p.~78]{Shanks1964} and Flajolet and Vardi \cite[pp.~7--9]{flajolet1996zeta} (compare \cite[Eq.~(3)]{Radziejewski}, \cite[Lem.~2.2]{Gorodetsky2021}) proved independently the identity
\begin{equation}\label{eq:geniden} 
\sum_{n \ge 1} \frac{\sosind(n)}{n^s} = \sqrt{\zeta(s)L(s,\quadratic)(1-2^{-s})^{-1}} \prod_{k \ge 1} \left(\frac{(1-2^{-2^{k}s})\zeta(2^k s)}{L(2^ks,\quadratic)}\right)^{2^{-k-1}},
\end{equation}
and their proof can yield an analogue of \eqref{eq:geniden} with a twist by $\chi(n)$.
Shanks and Flajolet and Vardi were interested in efficient computation of the constant $K$, and this identity leads to
\[ K=\frac{1}{\sqrt{2}}\prod_{k \ge 1}\left(\frac{\zeta(2^k)(1-2^{-2^k})}{L(2^k,\quadratic)}\right)^{2^{-k-1}}.\] 
Since both sides of \eqref{eq:geniden} enjoy Euler products, this identity can be verified by checking it locally at each prime; one needs to check $p=2$, $p\equiv 1 \bmod 4$ and $p \equiv 3 \bmod 4$ separately using \eqref{eq:characterization}. For the purpose of this paper we do not need the terms corresponding to $k>1$ in \eqref{eq:geniden}. What we need is stated and proved in Lemma~\ref{lem:analytic}, namely that
	\begin{equation}\label{eq:simpiden} 	F(s,\chi):=\sum_{n \ge 1} \frac{\sosind(n) \chi(n)}{n^s}  = \sqrt{L(s,\chi)L(s,\chi  \quadratic)} \sqrt[4]{\frac{L(2s,\chi^2)}{L(2s,\chi^2 \quadratic)}} G(s,\chi)
\end{equation}
for $G$ which is analytic and nonvanishing in $\Re s > 1/4$ and bounded in $\Re s \ge 1/4 +\varepsilon$ for each $\varepsilon>0$. The important feature of this formula is that it allows us to analytically continue $F(s,\chi)$ to the left of $\Re s =1/2$ (once we remove certain line segments), as opposed to \eqref{eq:landform} whose limit is $\Re s > 1/2$. See the discussion at the end \S\ref{sec:analyticcont}.

Recently, a formula very similar to \eqref{eq:simpiden} was used by Porritt \cite{Porritt2020} in his study of character sums over polynomials with $k$ prime factors, and $k$ tending to $\infty$. We state his formula in the integer setting. Let $\Omega$ be the additive function counting prime divisors with multiplicity. He showed that, for complex $z$ with $|z|<2$, we have \cite[Eq.~(4)]{Porritt2020}
\[ \sum_{n \ge 1} \frac{z^{\Omega(n)}\chi(n)}{n^s} = L(s,\chi)^z L(2s,\chi^2)^{\frac{z^2-z}{2}} E_{z}(s,\chi)\]
for $E_z(s,\chi)$ which is analytic in $\Re s >\max\{1/3, \log_2 |z|\}$. He then proceeds to apply a Selberg-Delange type analysis, leading to an explicit formula for a polynomial analogue of $\sum_{n \le x, \, \Omega(n)=k} \chi(n)$ where $k$ grows like $a\log \log x$ for $a \in (0,2^{1/2})$ (in the polynomial world, $q$ and $q^{1/2}$ replace $2$ and $2^{1/2}$, where $q$ is the size of the underlying finite field). His results show a strong Chebyshev's bias once $a>1.2021\ldots$ \cite[Thm.~4]{Porritt2020}.
\subsection{Analyzing singularities}
We shall analyze each of the sums in \eqref{eq:sdiff}. We first observe that we do not need to analyze the sums corresponding to  $\chi$ or $\chi\quadratic$ being principal, because these characters do not contribute to \eqref{eq:sdiff} (as $\chi_0(a)=\chi_0(b)$).

Assume GRH and let $\chi$ be a nonprincipal character such that $\chi \quadratic$ is also nonprincipal. We apply a truncated Perron's formula to $\sum_{n \le x}\sosind(n)\chi(n)$ (Corollary~\ref{cor:perrons}). We then want to shift the contour to $\Re s = 1/2 -c$ ($c=1/10$, say) and apply the Residue theorem. 
We cannot do it, because $L(s,\chi)$ and $L(s,\chi\quadratic)$ have zeros on $\Re s= 1/2$ so $F(s,\chi)$, which involves the square root of $L(s,\chi)L(s,\chi\quadratic)$, cannot be analytically continued to $\Re s \ge 1/2 - c$. Let us analyze the zeros and poles, in the half-plane $\Re s > 1/4$, of $L(s,\chi)$, $L(s,\chi \quadratic)$, $L(2s,\chi^2)$ and $L(2s,\chi^2 \quadratic)$. These are the functions which appear in \eqref{eq:simpiden} and dictate the region to which we may analytically continue $F(s,\chi)$.
\begin{itemize}
	\item We have zeros of $L(s,\chi)$, $L(s,\chi\quadratic)$ on $\Re s = 1/2$ (only) by GRH. We do not have poles at $s=1$ because we assume $\chi$, $\chi\quadratic$ are nonprincipal.
	\item Under GRH, $L(2s,\chi^2)$ and $L(2s,\chi^2 \quadratic)$ have no zeros in $\Re s >1/4$. 
	\item If $\chi^2$ is principal then $L(2s,\chi^2)$ has a simple pole at $s=1/2$. Similarly, if $\chi^2 \quadratic$ is principal then $L(2s,\chi^2 \quadratic)$ has a simple pole at $s=1/2$. If $\chi^2$ and $\chi^2 \quadratic$ are nonprincipal then these $L$-functions have no poles. 
\end{itemize}
We call these zeros and poles `singularities of $F$'. They all lie on $\Re s = 1/2$. We construct an open set $A_{\chi}$ by taking the half-plane $\Re s > 1/4$ are removing the segments $\{\sigma+it:1/4<\sigma \le 1/2\}$ for every singularity $1/2+it$. This domain is simply connected and $L(s,\chi)$, $L(s,\chi\quadratic)$, $L(2s,\chi^2)$ and $L(2s,\chi^2\quadratic)$ have no poles or zeros there. Hence, they have well-defined logarithms there and we may analytically continue $F(s,\chi)$ to $A_{\chi}$.

Although we cannot literally shift the contour to the left of $\Re s = 1/2$, we can move to a contour which stays in $A_{\chi}$ and is to the left of $\Re s = 1/2$ `most of the time'. Specifically, we shall use truncated Hankel loop contours going around the singularities, joined to each other vertically on $\Re s = 1/2-c$, as in Meng \cite{Meng2018,Meng2020}. The precise contour is described in \S\ref{sec:contour}. See Figure~\ref{fig:Contour} for depiction.
A truncated Hankel loop contour around a singularity $\rho$ of $F(s,\chi)$ is a contour $\mathcal{H}_{\rho}$ traversing the path depicted in Figure~\ref{fig:Hankel}. It is parametrized in \eqref{eq:hankel}. 
\begin{figure}
	\centering
	\begin{minipage}{.5\textwidth}
		\centering
		\begin{tikzpicture}
			\def\N{3}
	\def\gap{0.1}
	\def\delta{0.2}
	\def\radius{0.15}
	\def\epsilon{0.02}
	\def\cright{0.42}
	\def\c{0.2}
	\def\T{1.25}
	\def\kap{1.5}
	\def\cons{2*\delta+2*\epsilon}
	\def\n1{asin(\epsilon/\radius)}
	\def\arroww{0.5pt}
	\def\Circs{5}
	
	\draw[line width=0.6pt,->](-0.1,0) -- (2*\N,0);
	\draw[line width=0.6pt,->] (0,-1.5*\N) -- (0,1.5*\N);
	
	\draw[line width=0.8pt,decoration={markings,
		mark=at position 0.07 with{\arrow[line width =\arroww]{>}},
		mark=at position 0.27 with{\arrow[line width =\arroww]{>}},
		mark=at position 0.47 with{\arrow[line width=\arroww]{>}},
		mark=at position 0.53 with{\arrow[line width=\arroww]{>}},
		mark=at position 0.65 with {\arrow[line width =\arroww]{>}},
		mark=at position 0.8 with{\arrow[line width =\arroww]{>}},
		mark=at position 0.955 with{\arrow[line width =\arroww]{>}}},
	postaction={decorate}]
	(\N*\c,-\N*\T+\N*\Circs*\cons+2*\N*\delta+2*\N*\epsilon)--(\N*\kap,-\N*\T+\N*\Circs*\cons+2*\N*\delta+2*\N*\epsilon)
	(\N*\kap,\N*0*\cons-\N*\T)--(\N*\c,\N*0*\cons-\N*\T);

	\foreach \x in {0,...,\Circs}
	{\draw[line width=0.8pt,decoration={markings,
			mark=at position 0.07 with{\arrow[line width =\arroww]{>}},
			mark=at position 0.22 with{\arrow[line width =\arroww]{>}},
			mark=at position 0.4 with{\arrow[line width =\arroww]{>}},
			mark=at position 0.63 with{\arrow[line width =\arroww]{>}},
			mark=at position 0.85 with{\arrow[line width =\arroww]{>}}},
		postaction={decorate}]
		(\N*\c,\N*\x*\cons-\N*\T)--(\N*\c,-\N*\T+\N*\delta+\N*\x*\cons)--(\N*\cright,-\N*\T+\N*\delta+\N*\x*\cons)
		(\N*\cright,-\N*\T+\N*\delta+\N*\x*\cons) arc[start angle=180+\n1, end angle=540-\n1,radius=\N*\radius]
		(\N*\cright,-\N*\T+\N*\delta+2*\N*\epsilon+\N*\x*\cons)--(\N*\c,-\N*\T+\N*\delta+2*\N*\epsilon+\N*\x*\cons)--(\N*\c,-\N*\T+\N*\x*\cons+2*\N*\delta+2*\N*\epsilon);}
			\end{tikzpicture}
		\captionof{figure}{Contour of integration}
		\label{fig:Contour}
	\end{minipage}%
	\begin{minipage}{.5\textwidth}
		\centering
		\begin{tikzpicture}
			\def\N{10}
\def\radius{0.15}
\def\epsilon{0.02}
\def\cright{0.42}
\def\c{0.2}
\def\n1{asin(\epsilon/\radius)}
\def\arroww{1pt}
\def\Circs{5}
\def\cosn1{cos(\n1)}

\draw[line width=1pt,decoration={markings,
		mark=at position 0.07 with{\arrow[line width =\arroww]{>}},
		mark=at position 0.22 with{\arrow[line width =\arroww]{>}},
		mark=at position 0.4 with{\arrow[line width =\arroww]{>}},
		mark=at position 0.63 with{\arrow[line width =\arroww]{>}},
		mark=at position 0.94 with{\arrow[line width =\arroww]{>}}},
	postaction={decorate}]
	(\N*\c,0)--(\N*\cright,0)
	(\N*\cright,0) arc[start angle=180+\n1, end angle=540-\n1,radius=\N*\radius]
	(\N*\cright,2*\N*\epsilon)--(\N*\c,2*\N*\epsilon);
    
    \draw[line width=1pt](\N*\cright+\N*\radius,\N*\epsilon)--++(45:\N*\radius);
    
\coordinate (H) at (\N*\cright+\N*\radius,\N*\epsilon); \node[below] at (H){\small\bf {$\rho$}};
\fill   (\N*\cright+\N*\radius,\N*\epsilon) circle (0.05);

\coordinate (Q) at (\N*\cright+\N*\radius*1.4,\N*\epsilon+\N*\radius*0.3); \node[below] at (Q){\small\bf {$r$}};

\end{tikzpicture}
		\captionof{figure}{Truncated Hankel loop contour around $\rho$ with radius $r$}
		\label{fig:Hankel}
	\end{minipage}
\end{figure}
Given a character $\chi$ and a singularity $\rho$ of $F(s,\chi)$ we let
\[ f(\rho,\chi,x) := \frac{1}{2\pi i} \int_{\mathcal{H}_{\rho}} F(s,\chi)x^s \frac{ds}{s}\]
be the Hankel contour integral around $\rho$. By analyticity, the value of $f(\rho,\chi,x)$ is independent of $r$ (once $r$ is small enough) and for our purposes we choose $r=o(1/\log x)$.
We end up obtaining
\begin{equation}\label{eq:bothcases} \sum_{n \le x} \sosind(n) \chi(n) \approx \sum_{\substack{\rho:\, L(\rho,\chi)L(\rho,\chi \quadratic)=0\\ \text{or }\rho=\frac{1}{2}}} f(\rho,\chi,x).
\end{equation}
See Lemma~\ref{lem:perronshifted} for a  statement formalizing \eqref{eq:bothcases} (in practice we sum only over zeros up to a certain height). If $F(s,\chi) \sim C(s-\rho)^{m}$ asymptotically as $s \to \rho^+$ (i.e. as $s-\rho$ tends to $0$ along the positive part of the real line) for one of the singularities $\rho$, it can be shown that 
\begin{equation}\label{eq:loopsize}
f(\rho,\chi,x) \sim \frac{C}{\Gamma(-m)}\frac{x^{\rho}}{\rho}(\log x)^{-1-m}	
\end{equation}
as $x \to \infty$. An informal way to see this is 
\[ \int_{\mathcal{H}_{\rho}} F(s,\chi) x^s \frac{ds}{s} \approx \frac{C}{\rho}\int_{\mathcal{H}_{\rho}} (s-\rho)^m x^s ds = \frac{Cx^{\rho}}{\rho} (\log x)^{-1-m} \int_{(\log x)(\mathcal{H}_{\rho}-\rho)} z^m e^z dz\]
and the last integral gives \eqref{eq:loopsize} by Hankel's original computation \cite[Thm.~II.0.17]{Tenenbaum2015}.\footnote{To justify the first passage it suffices to show $\int_{\mathcal{H}_{\rho}} |G(s)x^s| |ds| =o(x^{\Re \rho}(\log x)^{-1-m})$ for $G(s) = F(s,\chi)/s - C(s-\rho)^m/\rho$, which is possible in our case.}
In particular, from analyzing $\sqrt{L(s,\chi)L(s,\chi\quadratic)}$ we see that we have
\[ f(\rho,\chi,x) \sim c_{\rho}\frac{x^{\rho}}{\rho}(\log x)^{-1-\frac{m_\rho}{2}}, \qquad c_{\rho} = \frac{\lim_{s \to \rho^+} F(s,\chi)(s-\rho)^{-\frac{m_{\rho}}{2}}}{\Gamma(-\frac{m_{\rho}}{2})}, \]
for any given $\rho \neq 1/2$ where $m_{\rho}$ is the multiplicity of $\rho$ in $L(s,\chi)L(s,\chi\quadratic)$. As $m_{\rho} \ge 1$, we are led to think of $\sum_{ \rho \neq 1/2} f(\rho,\chi,x)$ as a quantity of order $\asymp \sqrt{x}(\log x)^{-3/2}$. However, we are not able to use \eqref{eq:loopsize} in order to bound $\sum_{\rho \neq 1/2} f(\rho,\chi,x)$ efficiently. We proceed via a different route and show 
\begin{equation}\label{eq:average} \frac{1}{X}\int_{X}^{2X} \left|\frac{\sum_{\rho \neq 1/2} f(\rho,\chi,x)}{\sqrt{x}(\log x)^{-3/2}} \right|^2 dx \ll 1 
\end{equation}
without making use of \eqref{eq:loopsize}. It follows that $\sum_{\rho\neq 1/2} f(\rho,\chi,x)\ll \sqrt{x}(\log x)^{-3/2}$ most of the time. This estimate is analogous to \eqref{eq:l2 primes} and its proof is similar too. We believe this sum is $ \ll_{\varepsilon}\sqrt{x}(\log x)^{\varepsilon-3/2}$ always but are not able to show this. This is similar to how GRH can show $\sum_{n \le x}  \Lambda(n)\chi(n) \ll \sqrt{x} (\log x)^2$, but the true size in this question is expected to be $\ll_{\varepsilon} \sqrt{x} (\log x)^{\varepsilon}$. Here one should think of $f(\rho,\chi,x)$ as an analogue of the expression $-x^{\rho}/\rho$ from the explicit formula. 

Finally, let us analyze the contribution of the (possible) singularity at $1/2$. If $L(1/2,\chi)L(1/2,\chi \quadratic)=0$ or if $\chi^2 \quadratic$ is principal we can show (e.g. using \eqref{eq:loopsize}) that $f(1/2,\chi,x) \ll \sqrt{x}(\log x)^{-5/4}$, see Lemma~\ref{lem:hankelhalf}. The constant $5/4$ arises from analyzing $\sqrt{L(s,\chi)L(s,\chi\quadratic)}$ and $\sqrt[4]{L(2s,\chi^2)/L(2s,\chi^2 \quadratic)}$ and applying \eqref{eq:loopsize} with $m \ge 1/4$. It follows from \eqref{eq:bothcases} and \eqref{eq:average} that
\[ \frac{1}{X}\int_{X}^{2X} \left| \frac{\sum_{n \le x}\sosind(n)\chi(n)}{\sqrt{x}(\log x)^{-5/4}} \right|^2 dx \ll 1\]
unless $\chi^2$ is principal and $L(1/2,\chi)L(1/2,\chi \quadratic)\neq 0$. This remaining case leads to the bias. Indeed, we have
\[ F(s,\chi) \sim C\left(s-\frac{1}{2}\right)^{-1/4}\]
as $s\to \frac{1}{2}^{+}$ because of the fourth root in $\sqrt[4]{L(2s,\chi^2)}$, and this allows us to show (Lemma~\ref{lem:asymphalf}) that
\[ f(1/2,\chi,x) \asymp \sqrt{x}(\log x)^{-3/4}. \]
This is bigger than the typical contribution of all the other singularities by $(\log x)^{1/2}$. 
In one line, the bias comes from the fact that the value of $m$ in the asymptotic relation $F(s,\chi) \sim C(s-\rho)^m$ ($s \to \rho^+$, $\rho$ a singularity) is minimized when $\chi^2$ is principal, $\rho=1/2$ and $L(1/2,\chi)L(1/2,\chi \quadratic)\neq0$, in which case $m=-1/4$.
	\subsection{Computational evidence}
We examine the bias for $q \in \{3,5,15\}$. For $q=3$ we only have two relevant residues modulo $q$, namely $1 \bmod 3$ -- quadratic, and $2 \bmod 3$ -- nonquadratic. We have $L(1/2,\chi) \approx 0.480$ for the unique nonprincipal Dirichlet character modulo $3$ and $L(1/2,\chi \quadratic) \approx 0.498$. Under GRH, Corollary~\ref{cor:quad} predicts $S(x;3,1)>S(x;3,2)$ for almost all $x$. Up to $10^8$, $96.8\%$ of the time $S(x;3,1)>S(x;3,2)$. See Figure~\ref{fig:mod3} for the (quite oscillatory) graph of $S(x;3,1)-S(x;3,2)$ up to $10^7$.

\begin{figure}[h!]
	\caption{Graph of $S(x;3,1)-S(x;3,2)$ up to $10^7$.}
	\centering
	\includegraphics[width=15cm]{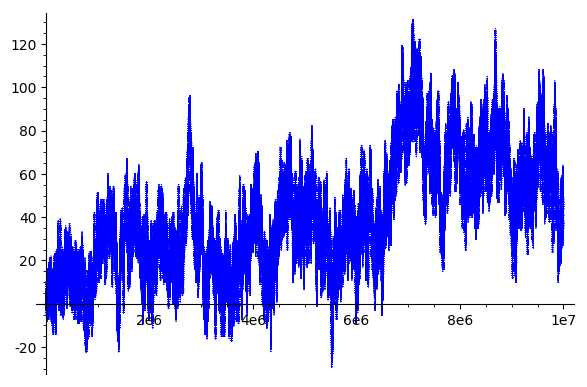}
	  \label{fig:mod3}
\end{figure}

For $q=5$ we have $4$ possible residues: $1$ and $4 \bmod 5$, both quadratic, and $2$ and $3 \bmod 5$, nonquadratic. We have $L(1/2,\chi)\approx 0.231$ and $L(1/2,\chi \quadratic)\approx 1.679$ for $\chi$, the unique nonprincipal quadratic character modulo $5$. 
 Corollary~\ref{cor:quad} predicts, under GRH, that $S(x;5,1)$ and $S(x;5,4)$ are almost always greater than $S(x;5,2)$ and $S(x;5,3)$.
The value of $C_{5,a,b}$ is simply $\chi(a)-\chi(b)$ times $(1+1/\sqrt{2})^{-1/2}\sqrt{L(1/2,\chi)L(1/2,\chi \quadratic)}$. For $(a,b) \in \{ (1,2),(1,3),(4,2),(4,3)\}$ its value is $\approx 0.7309$, since $\chi(a)=1$ and $\chi(b)=-1$. For $(a,b) \in \{ (1,2),(1,3),(4,2),(4,3)\}$, we find that the percentage of integers $n \le 10^7$ with $S(n;5,a)>S(n;5,b)$ is $96.1\%$, $95.2\%$, $95.3\%$ and $94.6\%$, respectively.

For $q=15$ we have the quadratic residues $1,4$ and the nonquadratic residues $2,7,8,11,13,14$. We expect $S(x;5,a)-S(x;5,b)$ to be positively biased for $a \in \{1,4\}$ and $b \in \{2,7,8,11,13,14\}$. Moreover, we expect a bias also for $S(x;5,a)-S(x;5,b)$ whenever $a\neq b \in \{2,7,8,11,13,14\}$ and $a/b \not\equiv 4 \bmod 15$, which means $(a,b) \neq (2,8)$, $(7,13)$, $(11,14)$. The direction of the bias in this case is harder to predict.
See Table~\ref{tab:1} for a table of the values of $C_{15,a,b}$ and Table~\ref{tab:2} for $\#\{n \le 10^7: S(n;15,a)>S(n;15,b)\}/10^7$ (in percentages). 
We omit pairs with $a \ge b$ due to symmetry and pairs $a,b$ with $a/b$ being a quadratic residue.
\begin{table}[H]
	\centering
\begin{tabular}{cccccccccc}
			&  & \multicolumn{8}{ c }{$b$} \\ 
			& & 1 & 4 & 2 & 7 & 8 & 11 & 13 & 14 \\
			\cline{3-10}
			\multicolumn{1}{ c  }{\multirow{7}{*}{$a$} } &
			\multicolumn{1}{ c| }{1} &  &  & 1.427 & 9.698 & 1.427 &9.931 & 9.698 & 9.931   
			\\ 
			\multicolumn{1}{ c  }{} &
			\multicolumn{1}{ c| }{4} &  &  & 1.427 &   9.698   & 1.427 & 9.931 & 9.698 &9.931  \\ 
			\multicolumn{1}{ c  }{} &
			\multicolumn{1}{ c| }{2} &  &  &   & 8.271  & & 8.504 & 8.271 & 8.504 \\ 
			\multicolumn{1}{ c  }{} &
			\multicolumn{1}{ c| }{7} &  &  &  & & -8.271 & 0.233 &  & 0.233 \\    
			\multicolumn{1}{ c  }{} &
			\multicolumn{1}{ c| }{8} &  &  &  & & & 8.504 & 8.271 & 8.504\\
						\multicolumn{1}{ c  }{} &
			\multicolumn{1}{ c| }{11} &  &  &  & & & &  -0.233 &\\
						\multicolumn{1}{ c  }{} &
			\multicolumn{1}{ c| }{13} &  &  &  & & & &  & 0.233
		\end{tabular}
\caption{Values of $C_{15,a,b}$.}
\label{tab:1}
\end{table}

\begin{table}[H]
	\centering
	\begin{tabular}{cccccccccc}
		&  & \multicolumn{8}{ c }{$b$} \\ 
		& & 1 & 4 & 2 & 7 & 8 & 11 & 13 & 14 \\
		\cline{3-10}
		\multicolumn{1}{ c  }{\multirow{7}{*}{$a$} } &
		\multicolumn{1}{ c| }{1} &  &  & 93.99 & 99.99 & 86.12 & 99.98 & 99.99 & 99.99  
		\\ 
		\multicolumn{1}{ c  }{} &
		\multicolumn{1}{ c| }{4} &  &  & 96.28 &   99.97   & 90.72 & 99.99 & 99.99 & 99.96  \\ 
		\multicolumn{1}{ c  }{} &
		\multicolumn{1}{ c| }{2} &  &  &   & 99.90  & & 99.85 & 99.93 & 99.90 \\ 
		\multicolumn{1}{ c  }{} &
		\multicolumn{1}{ c| }{7} &  &  &  & & 0.03 & 57.99 &  & 57.99 \\    
		\multicolumn{1}{ c  }{} &
		\multicolumn{1}{ c| }{8} &  &  &  & & & 99.52 & 99.96 & 99.99\\
		\multicolumn{1}{ c  }{} &
		\multicolumn{1}{ c| }{11} &  &  &  & & & &  40.19 &\\
		\multicolumn{1}{ c  }{} &
		\multicolumn{1}{ c| }{13} &  &  &  & & & &  & 59.23
	\end{tabular}
	\caption{Percentage of $n \le 10^7$ with $S(n;15,a)>S(n;15,b)$.}
\label{tab:2}
\end{table}
We see that the two tables are correlated. This is not a coincidence: the proof of Theorem~\ref{thm:bias} actually shows that $S(x;q,a)-S(x;q,b) = \phi(q)^{-1}C_{q}C_{q,a,b}\sqrt{x}/(\log x)^{3/4} + E(x)$ where $C_q$ is positive constant depending only on $q$ and $E(x)$ is a function which, on average, is smaller than $\sqrt{x}/(\log x)^{3/4}$. Concretely, $(1/X)\int_{X}^{2X} |E(x)|^2 dx \ll X/(\log X)^{5/2}$. So, for most values of $x$, $S(x;q,a)-S(x;q,b)$ is proportional to $C_{q,a,b}$.
\subsection{Martin's conjecture}\label{sec:martin conj}
Let $\omega$ be the additive function counting prime divisors (without multiplicity). In \cite{Meng2020}, Meng states a conjecture of Greg Martin, motivated by numerical data, saying that
\[ \left\{ x: \sum_{\substack{n \le x\\ n \equiv 1 \bmod 4}} \omega(n) < \sum_{\substack{n \le x\\ n \equiv 3 \bmod 4}} \omega(n)\right\}\]
contains all sufficiently large $x$. Meng assumed GRH and GSH to prove that this set has logarithmic density $1$. He also obtains results for other moduli under Chowla's Conjecture, and studies an analogous problem with the completely additive function $\Omega$. Meng writes: `In order to prove the full conjecture, one may need to formulate new ideas and introduce more powerful tools to bound the error terms of the summatory functions' \cite[Rem.~4]{Meng2020}. We are able to prove a natural density version of Meng's result, making progress towards Martin's conjecture. We do not assume GSH.
\begin{thm}\label{thm:omega}
Fix a positive integer $q $. Assume that GRH holds for the Dirichlet $L$-functions $L(s,\chi)$ for all Dirichlet character $\chi$ modulo $q$.		
	Then, whenever $a,b$ satisfy $(a,q)=(b,q)=1$ and \begin{equation}\label{eq:lincondomega} D_{q,a,b}:=\sum_{\substack{\chi \bmod q\\ \chi^2 = \chi_0}} (\chi(a)-\chi(b)) L\left( \frac{1}{2},\chi\right)>0
	\end{equation}
	we have, as $x \to \infty$,
	\[ \# \left\{n \le x: \sum_{\substack{ m \le n \\ m \equiv a \bmod q}} \omega(m)< \sum_{\substack{ m \le n \\ m \equiv b \bmod q}} \omega(m)\right\} = x(1+o(1)) \]
	and 
	\[ \# \left\{n \le x: \sum_{\substack{ m \le n \\ m \equiv a \bmod q}} \Omega(m)> \sum_{\substack{ m \le n \\ m \equiv b \bmod q}} \Omega(m)\right\} = x(1+o(1))\]
\end{thm} 
The proof is given in \S\ref{sec:omega}. If $a$ is a quadratic residue modulo $q$ and $b$ is not, a sufficient condition for $D_{q,a,b}$ to be positive is Chowla's Conjecture.
	
	\section{Preparatory lemmas}\label{sec:prep}
	Given a Dirichlet character $\chi$ modulo $q$ we write
	\[ F(s,\chi) = \sum_{n \in S} \frac{\chi(n)}{n^s} = \prod_{p \not \equiv 3 \bmod 4}(1-\chi(p)p^{-s})^{-1} \prod_{p \equiv 3 \bmod 4}(1-\chi^2(p)p^{-2s})^{-1}.\]
	This converges absolutely for $\Re s > 1$. For $\Re s = 1$ (or smaller) it does not, because such convergence implies $\sum_{p \not \equiv 3 \bmod 4,\, p \nmid q}1/p$ converges. Observe that $F(s,\chi)$ does not vanish for $\Re s > 1$.
	
	We shall use the convention where $\sigma$ and $t$ denote the real and complex parts of $s \in \CC$. 
	\subsection{Perron}
	\begin{lem}[Effective Perron's formula]\cite[Thm.~II.2.3]{Tenenbaum2015}\label{lem:effective0}
		Let $F(s)=\sum_{n \ge 1} a_n/n^s$ be a Dirichlet series with abscissa of absolute convergence $\sigma_a<\infty$. For $\kappa>\max\{0,\sigma_a\}$, $T \ge 1$ and $x \ge 1$ we have
		\begin{equation}
			\sum_{n \le x} a_n = \frac{1}{2\pi i} \int_{\kappa-iT}^{\kappa+iT} F(s)x^s \frac{ds}{s} + O\left(x^{\kappa}\sum_{n \ge 1} \frac{|a_n|}{n^{\kappa}(1+T|\log(x/n)|)}\right),
		\end{equation}
		with an absolute implied constant.
	\end{lem}
	This lemma leads to the following, which is a variation on \cite[Cor.~II.2.4]{Tenenbaum2015}.
	\begin{cor}\label{cor:effbdd}
		Suppose $|a_n| \le 1$. Let $F(s)=\sum_{n \ge 1} a_n/n^s$. Then, for $x,T \gg 1$,
		\begin{equation}
			\sum_{n \le x} a_n = \frac{1}{2\pi i} \int_{1+\frac{1}{\log x}-iT}^{1+\frac{1}{\log x}+iT} F(s)x^s \frac{ds}{s} + O\left(1 + \frac{x \log x}{T}\right),	
		\end{equation}
		with an absolute implied constant.
	\end{cor}
	\begin{proof}
		Since $|a_n|\le 1$ we have $\sigma_a \le 1$ so that we may apply Lemma~\ref{lem:effective0} with $\kappa = 1 + 1/\log x$.  The contribution of $n \le x/2$ to the error term is 
		\[\ll x \sum_{n \le x/2} \frac{1}{Tn} \ll \frac{x \log x}{T}.\]
		The contribution of  $n \ge 2x$ to the error is 
		\[ \ll x \sum_{k \ge 1} \sum_{n \in [2^k x, 2^{k+1} x)} \frac{1}{n^{\kappa}Tk} \ll x \sum_{k \ge 1} \frac{2^{-\frac{k}{\log x}}}{T} \ll \frac{x \log x}{T}.\]
		Finally, if $n \in (x/2,2x)$, the contribution is
		\[ \ll \sum_{n \in (x/2,2x)} \frac{1}{1+T|\log(x/n)|} \ll 1+\frac{x \log x}{T}\]
		where the second inequality follows e.g. by the argument in \cite[Cor.~II.2.4]{Tenenbaum2015}.
	\end{proof}
	As a special case of this corollary we have
	\begin{cor}\label{cor:perrons}
		Let $\chi$ be a Dirichlet character. We have
		\begin{equation}\label{eq:applicableperron} \sum_{\substack{n \le x\\ n \in S}}\chi(n) = \frac{1}{2\pi i}\int_{1+\frac{1}{\log x}-iT}^{1+\frac{1}{\log x}+iT} F(s,\chi)x^s \frac{ds}{s} + O\left( 1+\frac{x\log x}{T} \right).
		\end{equation}
	\end{cor}
	\begin{remark}\label{rem:modif}
		Since $|F(s,\chi)| \le \sum_{n \ge 1} n^{-1-1/\log x} \ll \log x$ when $\Re s = 1+1/\log x$, we see that perturbing the parameter $T$ (appearing in the range of integration) by $O(1)$ incurs an error of $O(x \log x/T)$ which is absorbed in the existing error term.
	\end{remark}	
	
	\subsection{Analytic continuation}\label{sec:analyticcont}
	Given a Dirichlet series $G(s)$ associated with a multiplicative function $g(n)$, which converges absolutely for $\Re s > 1$ and does not vanish there, we define the $k$th root of $G$ as
	\[ G^{1/k}(s)=\exp\left( \frac{\log G(s)}{k}\right)\]
	for each positive integer $k$, where the logarithm is chosen so that $\arg G(s) \to 0$ as $s \to \infty$.
	The function $G^{1/k}$ is also a Dirichlet series, since
	\[ G^{1/k}(s) = \prod_{p} \exp\left( \frac{\log \sum_{i \ge 0} \frac{g(p^i)}{p^{is}}}{k}\right). \]
	\begin{lem}\label{lem:analytic}
		Let $\chi$ be a Dirichlet character modulo $q$. We have, for $\Re s > 1$,
		\[ F(s,\chi) = \sqrt{L(s,\chi)L(s,\chi  \quadratic)} \sqrt[4]{\frac{L(2s,\chi^2)}{L(2s,\chi^2 \quadratic)}} G(s,\chi)\]
		for $G$ which is analytic and nonvanishing in $\Re s > 1/4$ and bounded in $\Re s \ge 1/4 +\varepsilon$.
		If $\chi$ is real then
		\begin{equation}\label{eq:Ghalf} G\left( \frac{1}{2},\chi\right)= \left(1-\frac{\chi(2)}{\sqrt{2}}\right)^{-1/2}\left(1-\frac{\mathbf{1}_{2 \nmid q}}{2}\right)^{1/4}\prod_{\substack{p\equiv 3 \bmod 4\\ p \nmid q}} \left(1-\frac{1}{p^2}\right)^{-1/4}.
		\end{equation}
	\end{lem}
	\begin{proof}
		We have, for $\Re s > 1$,
		\begin{align} \log G(s,\chi) &= \log F(s,\chi) - \frac{\log L(s,\chi)}{2} - \frac{\log L(s,\chi \quadratic)}{2} + \frac{\log L(2s,\chi^2 \quadratic)}{4} - \frac{\log L(2s,\chi^2)}{4} \\
			\label{eq:logGdef}&= \frac{-\log(1-\chi(2)2^{-s})}{2} + \frac{\log(1-\chi^2(2)2^{-2s})}{4} + \sum_{p \equiv 3 \bmod 4} g_p(s,\chi)
		\end{align}
		for
		\[ g_p(s,\chi) = \frac{-\log(1-\chi^2(p)p^{-2s})}{2} + \frac{\log(1-\chi^2(p)p^{-2s})-\log(1+\chi^2(p)p^{-2s})}{4}.\]
		Each $g_p$ is analytic in $\Re s > 0$. Fix $c>0$. For $\Re s \ge c$ we have
		\[ g_p(s,\chi) = \frac{\chi^4(p)}{4}p^{-4s} + O_c(p^{-6c}) = O_c(p^{-4c})\]
		by Taylor expanding $\log(1+x)$. In particular, if $c=1/4+\varepsilon$ we have
		\[ \left|\sum_{p \equiv 3 \bmod 4} g_p(s,\chi)\right| \ll \sum_{p \equiv 3 \bmod 4} p^{-4c} \ll \sum_{n \ge 1} n^{-1-4\varepsilon} < \infty \]
		and so $G$ may be extended to $\Re s > 1/4$ via \eqref{eq:logGdef}.
		As each $g_p$ is analytic in $\Re s > 1/4$, and $\sum_{p \equiv 3 \bmod 4} g_p(s,\chi)$ converges uniformly in $\Re s \ge 1/4+\varepsilon$ for every choice of $\varepsilon>0$, it follows that $\log G$ is analytic in $\Re s > 1/4$ and so is $G$. The formula for $G(1/2,\chi)$ for real $\chi$ follows from evaluating $g_p$ at $s=1/2$ and observing $\chi^2 =\chi_0$.
	\end{proof}
	Assuming GRH for $L(s,\chi)$ where $\chi$ is nonprincipal, $L(s,\chi)^{1/k}$ may be analytically continued to the region
	\begin{equation}\label{eq:analyticreg} \CC \setminus (\{ \sigma+it :  L(1/2+it,\chi)=0 \text{ and }\sigma \le 1/2 \} \cup \{ \sigma: \sigma \le -1\})
	\end{equation}
	and is nonzero there (because of trivial zeros of $L(s,\chi)$ we have to remove $\{\sigma:\sigma\le -1\}$). For $\chi$ principal we have a singularity at $s=1$ and so 
	$L(s,\chi)^{1/k}$ may be analytically continued to 
	\begin{equation}\label{eq:analyticreg2} \CC \setminus \left(\{ \sigma+it :  L(1/2+it,\chi)=0 \text{ and }\sigma \le 1/2 \} \cup \{ \sigma: \sigma \le 1\}\right).
	\end{equation}
	Hence, given $\chi$ which is nonprincipal and such that $\chi \quadratic$ is nonprincipal, we have the following.
	Under GRH for $\chi$, $\chi \quadratic$, $\chi^2$ and $\chi^2 \quadratic$, we may continue $F(s,\chi)$ analytically to
	\begin{equation}\label{eq:analyticcomplex}\left\{ s \in \CC: \Re s > \frac{1}{4}\right\} \setminus \{ \sigma+it :  L(1/2+it,\chi)L(1/2+it,\chi  \quadratic)=0 \text{ and }\sigma \le 1/2 \}
	\end{equation}
	if both $\chi^2$ and $\chi^2 \quadratic$ are nonprincipal; otherwise we may continue it to  
	\begin{equation}\label{eq:analyticreal}\left\{ s \in \CC: \Re s > \frac{1}{4}\right\} \setminus \left(\{ \sigma+it :  L(1/2+it,\chi)L(1/2+it,\chi  \quadratic)=0 \text{ and }\sigma \le 1/2 \} \cup \{ \sigma: \sigma \le 1/2\}\right).
	\end{equation}
	\subsection{$L$-function estimates}
	We quote three classical bounds on $L$-functions from the book of Montgomery and Vaughan \cite{Montgomery2007}.
	\begin{lem}\label{lem:upper}\cite[Thm.~13.18 and Ex.~8 at \S13.2.1]{Montgomery2007}
		Let $\chi$ be a Dirichlet character. Under GRH for $L(s,\chi)$, there exists a constant $A$ depending only on $\chi$ such that the following holds. Uniformly for $\sigma \ge 1/2$ and $|t| \ge 1$,
		\[ |L(s,\chi)| \le \exp\left( A \frac{\log (|t|+4)}{\log \log (|t|+4)}\right).\]
	\end{lem}

		\begin{lem}\label{lem:lower}\cite[Thm.~13.23]{Montgomery2007}
	Let $\chi$ be a Dirichlet character. Suppose $|t| \gg 1$. Under GRH for $L(s,\chi)$, there exists a constant $A$ depending only on $\chi$ such that the following holds. Uniformly for $\sigma \ge 1/2 + 1/\log \log (|t|+4)$ and $|t| \ge 1$,
	\[ \left| \frac{1}{L(s,\chi)}\right| \le \exp\left( A \frac{\log(|t|+4)}{\log \log (|t|+4)}\right).\]
\end{lem}
	\begin{lem}\label{lem:log}\cite[Cor.~13.16 and Ex.~6(c) at \S13.2.1]{Montgomery2007}
		Let $\sigma \in (1/2,1)$ be fixed. Let $\chi$ be a primitive Dirichlet character. We have, as $|t| \to \infty$, \[|\log L(s,\chi)| \ll_{\sigma} \frac{(\log(|t|+4) )^{2-2\sigma}}{\log \log (|t|+4)}.\]
	\end{lem}
	The following is a consequence of the functional equation.
\begin{lem}\label{lem:fe}\cite[Cor.~10.10]{Montgomery2007}
	Let $\chi$ be a Dirichlet character and $\varepsilon \in (0,1)$. We have $|L(s,\chi)| \asymp |L(1-s,\overline{\chi})| (|t|+4)^{1/2-\sigma}$ uniformly for $\varepsilon \le \sigma \le 1/2$ and $|t|\ge 1$, where the implied constants depend only on $\chi$ and $\varepsilon$.
\end{lem}
These four lemmas are originally stated for primitive characters. However, if $\chi$ is induced from a primitive character $\psi$, then in $\Re s >0$ the ratio $L(s,\chi)/L(s,\psi)$ is equal to the finite Euler product $\prod_{p:\, \chi(p)=0}(1-\psi(p)/p^s)$. This product is bounded away from $0$ and from $\infty$ when $\Re s \ge \varepsilon$, so we can convert results for $L(s,\psi)$ to results for $L(s,\chi)$ as long as we restrict our attention to $\sigma \ge \varepsilon$. 
	\subsection{Contour choice}\label{sec:contour}
	Let $\chi$ be a nonprincipal Dirichlet character modulo $q$. Fix $c \in (0,1/8)$ (say, $c=1/10$). Let $T \gg 1$. 
	
	We want to use Cauchy's Integral Theorem to shift the vertical contour appearing in \eqref{eq:applicableperron} to the left of $\Re s = 1/2$ (namely to $\Re s = 1/2 - c$), at the `cost' of certain horizontal contributions. As we want to avoid zeros of $L(s,\chi)L(s,\chi \quadratic)$ and poles and zeros of $L(2s,\chi^2)/L(2s,\chi^2\quadratic)$ (which by GRH can only occur at $s=1/2$), we will use (truncated) Hankel loop contours to go around the relevant zeros and poles; the integrals over these loops will be the main contribution to our sum. It will also be convenient for $1/2- iT$ and $1/2+iT$ to avoid zeros of $L(s,\chi)L(s,\chi\quadratic)$; this is easy due to Remark~\ref{rem:modif}, showing that changing $T$ by $O(1)$ does not increase the error term arising from applying Perron's (truncated) formula. We replace the range $[1+1/\log x - iT,1+1/\log x+iT]$ with $[1+1/\log x-iT'', 1+1/\log x + i T']$ where $T-1 \le T',T'' \le T$ and \[L\left(\frac{1}{2}+it,\chi\right)L\left(\frac{1}{2}+it,\chi \quadratic\right) \neq 0\]
	for every $t \in [T',T) \cup (-T,-T'']$.	
	Let \begin{equation}\label{eq:zeros}
		\gamma_1<\gamma_2<\ldots < \gamma_m
	\end{equation}
	be the imaginary parts of the zeros of  $L(s,\chi)L(s,\chi \quadratic)$ on $\sigma = 1/2$ with $t\in (-T,T)$ (without multiplicities), and, if either $\chi^2$ or $\chi^2 \quadratic$ is principal, we include the number $0$ (if it is not there already). Let $r \in (0,1)$ be a parameter that will tend to $0$ later. Consider the contour
	\begin{equation}\label{eq:contour} I_1 \cup \bigcup_{j=1}^{m} (J_j \cup \mathcal{H}_{1/2+i\gamma_j}) \cup J_{m+1} \cup I_2
	\end{equation}
	where $I_1$ traverses the horizontal segment
	\[ I_1 = \left\{ \sigma - i T'' : \frac{1}{2}-c \le \sigma \le 1+\frac{1}{\log x}\right\}\]
	from right to left, $I_2$ traverses the horizontal segment
	\[ I_2 = \left\{ \sigma + i T' : \frac{1}{2}-c \le \sigma \le 1+\frac{1}{\log x}\right\}\]
	from left to right, $J_j$ traverses the following vertical segment from its bottom point to the top:
	\[ J_j = \left\{ \frac{1}{2}-c +it:  \gamma_{j-1} \le t \le \gamma_{j} \right\}\]
	where 
	\[ \gamma_0 := -T'', \qquad \gamma_{m+1}:=  T',\]
	and finally each $\mathcal{H}_{\rho}$ traverses the following truncated Hankel loop contour in an anticlockwise fashion:
	\begin{multline}\label{eq:hankel} 
	\left\{ s \in \CC: \frac{1}{2}-c \le \Re s \le 	\frac{1}{2}-r, \, \Im s =\Im \rho, \, \arg(s-\rho)=-\pi\right\}  \\ \cup \left\{ s\in \CC: |s-\rho|=r, \, -\pi < \arg(s-\rho)<\pi\right\}  \cup \left\{ \frac{1}{2}-c \le \Re s \le \frac{1}{2}-r, \Im s= \Im \rho, \, \arg(s-\rho)=\pi \right\}
	\end{multline}
	where in our case $c=1/10$ and $r=o(1/\log x)$.	We refer the reader to Tenenbaum \cite[pp.~179--180]{Tenenbaum2015} for background on the Hankel contour and its truncated version.
	If $r$ is small enough, the contour in \eqref{eq:contour} does not intersect itself. 
		
	If both $\chi^2$ and $\chi^2 \quadratic$ are nonprincipal characters and the corresponding $L$-functions satisfy GRH
	observe $\sqrt[4]{L(2s,\chi^2)/L(2s,\chi^2\quadratic)}$ is analytic in $\Re s > 1/2 - 2c > 1/4$ and so is $F(s,\chi)$ by Lemma~\ref{lem:analytic}. 
	
	If $\chi^2$ is principal then $\chi^2  \quadratic$ cannot be principal. Similarly, if $\chi^2  \quadratic$ is principal then $\chi$ is a nonreal Dirichlet character of order $4$ and $\chi^2$ cannot be principal. In both cases, $\sqrt[4]{L(2s,\chi^2)/L(2s,\chi^2\quadratic)}$ has an algebraic singularity at $s=1/2$, which we avoid already as we inserted $0$ to the list \eqref{eq:zeros} if it is not there already. 
	 
	 In any case, by Cauchy's Integral Theorem, 
	 \begin{equation}\label{eq:cauchy}
	 	\frac{1}{2\pi i}\int_{1+\frac{1}{\log x}-iT''}^{1+\frac{1}{\log x}+iT'} F(s,\chi)x^s \frac{ds}{s} = \frac{1}{2\pi i}\left( \int_{I_1} + \sum_{j=1}^{m}\left(\int_{J_j} + \int_{\mathcal{H}_{1/2+i\gamma_j}}\right) + \int_{J_{m+1}} + \int_{I_2} \right) F(s,\chi)x^s \frac{ds}{s}.
	 \end{equation} 
 	\begin{lem}\label{lem:perronshifted}
		Let $\chi$ be a nonprincipal Dirichlet character. Assume GRH holds for the following four characters:
		\begin{equation}\label{eq:fourchars}
		\chi, \qquad \chi \quadratic,\qquad  \chi^2, \qquad \chi^2 \quadratic.
	\end{equation}
		Let $c \in (0,1/8)$ be a fixed constant. Let $T \gg 1$. We have
		\begin{multline} \sum_{\substack{n \le x \\ n \in S}} \chi(n) = \sum_{j=1}^{m} \frac{1}{2\pi i}\int_{\mathcal{H}_{1/2+i\gamma_j}} F(s,\chi)x^s \frac{ds}{s} \\+ O\left( 1 + \frac{x\log x}{T}+\frac{(x+\max\{x,T\}^{1/2}+x^{\frac{1}{2}-c}T^{c+1})\exp\left( \frac{A \log T}{\log \log T}\right)}{T}\right).
		\end{multline}
The implied constant and $A$ depend only on $\chi$ and $c$.
	\end{lem}
	\begin{proof}
By \eqref{eq:applicableperron}, Remark~\ref{rem:modif} and 
		\eqref{eq:cauchy}, it suffices to upper bound $\int_{I_j}|F(s,\chi)| |x^s| |ds|/|s|$ and $\int_{J_j}|F(s,\chi)| |x^s| |ds|/|s|$. We first treat $I_j$, and concentrate on $I_2$ as the argument for $I_1$ is analogous. We have, using Lemma~\ref{lem:analytic},
		\begin{align} \int_{I_2}|F(s,\chi)| |x^s| \frac{|ds|}{|s|} &\ll \frac{1}{T}\int_{\frac{1}{2}-c}^{1+\frac{1}{\log x}} \sqrt{|L(\sigma+iT',\chi)L(\sigma+iT', \chi \quadratic)|} \sqrt[4]{\frac{|L(2\sigma+2iT',\chi^2)|}{|L(2\sigma+2iT',\chi^2\quadratic)|}} x^{\sigma}d\sigma\\
			&\ll \frac{1}{T}\int_{\frac{1}{2}-c}^{1+\frac{1}{\log x}} \sqrt{|L(\sigma+iT',\chi)L(\sigma+iT', \chi \quadratic)|} \sqrt[4]{\frac{|L(2\sigma+2iT',\chi^2)|}{|L(2\sigma+2iT',\chi^2\quadratic)|}} x^{\sigma}d\sigma.
		\end{align}
		It is now convenient to consider $\sigma \ge 1/2$ and $\sigma \le 1/2$ separately.
		
		If $\sigma \ge 1/2$ we bound all the relevant $L$-functions using Lemmas~\ref{lem:upper} and \ref{lem:lower}, obtaining that this part of the integral contributes
		\[ \ll \frac{\exp\left( \frac{A \log T}{\log \log T}\right)}{T} \int_{\frac{1}{2}}^{1+\frac{1}{\log x}} x^{\sigma} d\sigma \ll \frac{x\exp\left( \frac{A \log T}{\log \log T}\right)}{T}\]
		where $A$ is a constant large enough depending on $\chi$.
		For $\sigma$ below $1/2$ we first apply Lemma~\ref{lem:fe} to the $L$-functions of $\chi$ and $\chi \quadratic$ to reduce to the situation where the real parts of the variables inside the $L$-functions are $\ge 1/2$. Then we apply Lemmas~\ref{lem:upper} and \ref{lem:lower} as before, obtaining that this part of the integral contributes
		\[ \ll \frac{\exp\left( \frac{A \log T}{\log \log T}\right)}{T} \int_{\frac{1}{2}-c}^{\frac{1}{2}} T^{1/2-\sigma}x^{\sigma} d\sigma \ll \frac{\max\{x,T\}^{1/2}\exp\left( \frac{A \log T}{\log \log T}\right)}{T}.\]
		It follows that
		\[ \left(\int_{I_1} + \int_{I_2}\right)|F(s,\chi)| |x^s| \frac{|ds|}{|s|} \ll \frac{(x+\max\{x,T\}^{1/2})\exp\left( \frac{A \log T}{\log \log T}\right)}{T}. \]
		We turn to the contribution of $J_j$. We have
		\begin{align} \sum_{j=1}^{m+1} \int_{J_j} |F(s,\chi)| |x^s| \frac{|ds|}{|s|} &\ll x^{\frac{1}{2}-c}\int_{-T-1}^{T+1} \left| F\left( \frac{1}{2}-c+it,\chi\right)\right| \frac{dt}{|t|+1}\\ 
			& \ll x^{\frac{1}{2}-c}\left( 1 + \sum_{2^k \le 2T}2^{-k} \left(\int_{2^{k-1}}^{2^{k}}+\int_{-2^{k}}^{-2^{k-1}}\right) \left| F\left( \frac{1}{2}-c+it,\chi\right)\right|dt\right)
		\end{align}
	where
	\begin{align} \left| F\left( \frac{1}{2}-c+it,\chi\right)\right| &= \sqrt{\left|L\left(\frac{1}{2}-c+it,\chi\right)L\left(\frac{1}{2}-c+it, \chi \quadratic\right)\right|} \sqrt[4]{\frac{|L(1-2c+2it,\chi^2)|}{|L(1-2c+2it,\chi^2\quadratic)|}} \\
		&\ll  (|t|+4)^{c} \exp\left( A \frac{\log(|t|+4)}{\log \log (|t|+4)}\right) \ll T^{c} \exp\left(  \frac{A\log T}{\log \log T}\right),
	\end{align}
where we used Lemmas~\ref{lem:upper} and \ref{lem:fe} to bound the $L$-functions of $\chi$ and $\chi \quadratic$ and Lemma~\ref{lem:log} for the other two $L$-functions. This leads to 
\[ \sum_{j=1}^{m+1} \int_{J_j} |F(s,\chi)| |x^s| \frac{|ds|}{|s|}  \ll x^{\frac{1}{2}-c} T^{c} \exp\left( \frac{A\log T}{\log \log T}\right)\]
and concludes the proof.
\end{proof}

\section{Hankel calculus}
In this section, $\mathcal{H}_{\rho}$ is the Hankel contour described in \eqref{eq:hankel}, going around $\rho$ in an anticlockwise fashion.
\begin{lem}\label{lem:ratio}
	Let $\chi$ be a Dirichlet character. Assume GRH for $L(s,\chi)$. Given a nontrivial zero $\rho=1/2+i\gamma$ of $L(s,\chi)$ we have
	\[ \max_{s \in \mathcal{H}_{\rho}} \left| \frac{L(s,\chi)}{s-\rho}\right| \ll (|\gamma|+1)^{c+o(1)}.\]
	Here the $o(1)$ exponent goes to $0$ as $\gamma$ goes to $\infty$ (and it might depend on $\chi$), and the implied constant is absolute.
\end{lem}
\begin{proof}
Since we can write $L(s,\chi)$ as $L(\rho,\chi)$ plus an integral of $L'(z,\chi)$ over a line segment connecting $s$ and $\rho$, it follows that the maximum we try to bound is
\[\ll \max_{\substack{|s-\rho| \le r, \text{ or }\\ s=\sigma+i\gamma \text{ with}\\\frac{1}{2}-c \le \sigma \le \frac{1}{2}}} |L'(s,\chi)|. \]
By Cauchy's integral formula, Lemma~\ref{lem:upper} and Lemma~\ref{lem:fe},
\[ |L'(s,\chi)| \ll \int_{|z-s|=\frac{1}{\log (|\gamma|+1)}} \frac{|L(z,\chi)|}{|z-s|}|dz| \ll (|\gamma|+1)^{\max\{0,\frac{1}{2}-\Re s\} + o(1)},\]
implying the desired bound. 
\end{proof}

\begin{lem}\label{lem:hankelabsolute}
Let $\chi$ be a nonprincipal Dirichlet character. Assume GRH for the characters in \eqref{eq:fourchars}. Given a nontrivial zero $\rho=1/2+i\gamma \neq 1/2$ of $L(s,\chi)L(s,\chi \quadratic)$ we have
\[\int_{\mathcal{H}_{\rho}} |F(s,\chi)||x^s| |ds| \ll \sqrt{x} (|\gamma|+1)^{c+o(1)}\left( (\log x)^{-3/2} +  r^{3/2} x^r \right).\]
Here the $o(1)$ exponent goes to $0$ as $\gamma$ goes to $\infty$ (and might depend on $\chi$), and the implied constant depends only on $\chi$ (it is independent of $r$).
\end{lem}

\begin{proof}
We have, for $s$ on the contour,
\[ |F(s,\chi)| \ll \sqrt{|L(s,\chi)L(s,\chi \quadratic)|} (|\gamma|+1)^{o(1)} \]
by Lemmas~\ref{lem:analytic}, \ref{lem:upper} and \ref{lem:lower}.
Integrating $\sqrt{|L(s,\chi)L(s,\chi \quadratic)|}|x^{s}|$ over the circle part of the contour contributes
\[ \ll rx^{\frac{1}{2}+r}\int_{-\pi}^{\pi} \sqrt{\left|L\left(\rho+re^{i\theta},\chi\right)L\left(\rho+re^{i\theta},\chi\quadratic\right)\right|}d\theta.\]
Writing $L(s,\chi)L(s,\chi \quadratic)$ as $L(s,\chi)L(s,\chi \quadratic)/(s-\rho)$ times $s-\rho$ and appealing to Lemmas~\ref{lem:ratio}, \ref{lem:upper} and \ref{lem:fe}, we find that this is
\[ \ll rx^{\frac{1}{2}+r} (|\gamma|+1)^{c+o(1)}\int_{-\pi}^{\pi} \sqrt{|re^{i\theta}|}d\theta \ll r^{3/2}x^{\frac{1}{2}+r} (|\gamma|+1)^{c+o(1)}.\]

Integrating $\sqrt{|L(s,\chi)L(s,\chi \quadratic)|}|x^{s}|$ over one of the segment parts of the contour contributes
\[ \ll \sqrt{x}\int_{-c}^{0} \sqrt{\left|L\left(\rho+t,\chi\right)L\left(\rho+t,\chi\quadratic\right)\right|}x^{t}dt.\]
Again writing $L(s,\chi)L(s,\chi \quadratic)$ as $L(s,\chi)L(s,\chi \quadratic)/(s-\rho)$ times $s-\rho$ and appealing to Lemma~\ref{lem:ratio}, we can bound this contribution by
\begin{equation}
\ll \sqrt{x} (|\gamma|+1)^{c+o(1)} \int_{-\infty}^{0} \sqrt{|t|}x^t dt \ll \sqrt{x} (|\gamma|+1)^{c+o(1)} (\log x)^{-3/2},
\end{equation}
concluding the proof.
\end{proof}

\begin{lem}\label{lem:l2bound}
Let $\chi$ be a nonprincipal Dirichlet character. Assume GRH for the characters in \eqref{eq:fourchars}. For any pair $\rho_1=1/2+i\gamma_1$, $\rho_2=1/2+i\gamma_2$ of nontrivial zeros different from $1/2$ we have
\begin{equation} \frac{1}{X}\int_{X}^{2X} \int_{\mathcal{H}_{\rho_1}}\int_{\mathcal{H}_{\rho_2}} F(s_1,\chi) \overline{F(s_2,\chi)} \frac{x^{s_1}}{s_1}\frac{x^{\overline{s_2}}}{\overline{s_2}}ds_1 \overline{ds_2} dx\ll \frac{X \left( (\log X)^{-3/2} +  r^{3/2} X^r \right)^2}{((|\gamma_1|+1) (|\gamma_2|+1))^{1-c+o(1)}(1+|\gamma_1-\gamma_2|)}
\end{equation}
where implied constants depend only on $\chi$.
\end{lem}
\begin{proof}
We first integrate by the $x$-variable and then take absolute values, obtaining that the integral is 
\begin{align} &\ll  \int_{\mathcal{H}_{\rho_1}}\int_{\mathcal{H}_{\rho_2}} |F(s_1,\chi)| |F(s_2,\chi)| \frac{X^{\Re(s_1+s_2)}}{|s_1| |s_2| |s_1 + \overline{s_2}+1|}|ds_1| | ds_2| \\
&\ll \frac{1}{|\rho_1 \rho_2|(1+|\gamma_1-\gamma_2|)}\int_{\mathcal{H}_{\rho_1}}\int_{\mathcal{H}_{\rho_2}} |F(s_1,\chi)| |F(s_2,\chi)| X^{\Re(s_1+s_2)}|ds_1| | ds_2| \\
& = \frac{1}{|\rho_1 \rho_2|(1+|\gamma_1-\gamma_2|)}\left( \int_{\mathcal{H}_{\rho_1}}  |F(s_1,\chi)| X^{\Re(s_1)} |ds_1| \right) \left( \int_{\mathcal{H}_{\rho_2}}  |F(s_2,\chi)| X^{\Re(s_2)} |ds_2| \right) .
\end{align}
The result now follows from Lemma~\ref{lem:hankelabsolute}.
\end{proof}

\begin{lem}\label{lem:hankelhalf}
Let $\chi$ be a nonprincipal Dirichlet character. Assume GRH for the characters in \eqref{eq:fourchars}. If $\chi^2 \quadratic$ is principal, or if $\chi^2$ is principal as well as $L(1/2,\chi)L(1/2,\chi \quadratic)=0$ then 
	\[ \int_{\mathcal{H}_{1/2}} F(s,\chi)\frac{x^s}{s} ds \ll \sqrt{x} \left( x^r r^{5/4} + \frac{1}{(\log x)^{5/4}}\right).\]
The implied constant depends only on $\chi$ (it is independent of $r$).
\end{lem}
\begin{proof}
	This is a variation on Lemma~\ref{lem:hankelabsolute}. If $\chi^2 \quadratic$ is principal we have 
\[ |F(s,\chi)| \ll \sqrt[4]{|2s-1|}\]
on $\mathcal{H}_{1/2}$, where the implied constant depends on $\chi$ only. Similarly, if $\chi^2$ is principal and $L(1/2,\chi)L(1/2,\chi \quadratic)=0$ then
\[ |F(s,\chi)| \ll \sqrt{\left| s- \frac{1}{2}\right|} \sqrt[4]{\frac{1}{|2s-1|}}\le \sqrt[4]{|2s-1|}\]
by Lemma~\ref{lem:ratio}. In both cases the integral is 
\[ \ll \sqrt{x}\int_{\mathcal{H}_{1/2}} \sqrt[4]{\left| s- \frac{1}{2}\right|}x^{\max\{\Re s - \frac{1}{2},0\}} |ds|. \]
The contribution of $\Re s \ge 1/2$ is $\ll \sqrt{x}x^{r} r^{5/4}$, while the contribution of $\Re s \le 1/2$ is 
\[ \ll \sqrt{x} \int_{-\infty}^{0} \sqrt[4]{|t|}x^{t} dt \ll  \frac{\sqrt{x}}{(\log x)^{5/4}}, \]
concluding the proof. 
\end{proof}
Let $C_q$ be the following positive constant, depending only on $q$:
\begin{equation}\label{eq:cqdef} C_{q} = \frac{2\pi^{-1/4}}{\Gamma(1/4)} \prod_{p \equiv 3 \bmod 4} \left(1-\frac{1}{p^2}\right)^{-1/4} \prod_{\substack{p \mid q \\ p \equiv 3 \bmod 4}} \left(1 - \frac{1}{p}\right)^{1/2}.
\end{equation}
\begin{lem}\label{lem:asymphalf}
Let $\chi$ be a nonprincipal Dirichlet character modulo $q$. Assume GRH for the characters in \eqref{eq:fourchars}. If $\chi^2$ is principal and $L(1/2,\chi)L(1/2,\chi\quadratic)\neq 0$  then
	\begin{multline}\frac{1}{2\pi i}\int_{\mathcal{H}_{1/2}} F(s,\chi)\frac{x^s}{s} ds = C_{q}\frac{\sqrt{x}}{(\log x)^{3/4}} \left(1-\frac{\chi(2)}{\sqrt{2}}\right)^{-1/2}\sqrt{L\left(\frac{1}{2},\chi\right)L\left(\frac{1}{2},\chi \quadratic\right)} \left( 1+O\left(  x^{-c/2}\right)\right)
		 \\ + O\left( \sqrt{x}\left( x^r r^{7/4} + \frac{1}{(\log x)^{7/4}} \right) \right)
		\end{multline}
where the implied constants depend only on $\chi$.
\end{lem}
\begin{proof}
	Let $M(s,\chi) = \sqrt{L(s,\chi)L(s,\chi\quadratic)} \sqrt[4]{L(2s,\chi^2)(2s-1)/L(2s,\chi^2 \quadratic)}G(s,\chi)/s$ where $G$ is defined in Lemma~\ref{lem:analytic}. On $\mathcal{H}_{1/2}$ we have $F(s,\chi)/s=M(s,\chi)(2s-1)^{-1/4}$. We define $M$ at $s=1/2$ by its limit there, which exists as $L(s,\chi_0) = \zeta(s)\prod_{p\mid q}(1-p^{-s})$ has a simple pole at $s=1$. In fact, $M$ is analytic in a neighborhood of $s=1/2$ by our assumption on $\chi$ and $\quadratic$. We have
	\[ M\left(\frac{1}{2},\chi\right) = 2\sqrt{L\left(\frac{1}{2},\chi\right)L\left(\frac{1}{2},\chi \quadratic\right)} \prod_{p \mid q}\left(1-\frac{1}{p}\right)^{1/4}L(1,\chi_0 \quadratic)^{-1/4}G\left( \frac{1}{2},\chi\right). \]
	The expression $L(1,\chi_0\quadratic)$ may be simplified as $L(1,\quadratic) \prod_{p \mid q}(1-\quadratic(p)/p) = \prod_{p \mid q}(1-\quadratic(p)/p)\pi/ 4$.
Our integral is
\[\frac{1}{2\pi i}\left(M\left( \frac{1}{2},\chi \right) \int_{\mathcal{H}_{1/2}} (2s-1)^{-1/4}x^s ds + \int_{\mathcal{H}_{1/2}} (2s-1)^{-1/4}x^s \left( M(s,\chi)- M\left(\frac{1}{2},\chi\right)\right) ds\right).\]
The second integral here is small, namely $\ll\sqrt{x} \left( x^r r^{7/4} + (\log x)^{-7/4}\right)$, by an argument parallel to Lemma~\ref{lem:hankelhalf}. It suffices to show that
\[ \frac{1}{2\pi i}\int_{\mathcal{H}_{1/2}} \left(s-\frac{1}{2}\right)^{-1/4} x^s ds = \frac{1}{\Gamma(1/4)} \frac{\sqrt{x}}{(\log x)^{3/4}}  \left( 1+O\left(  x^{-c/2}\right)\right).\]
Making the change of variables $\left(s-\frac{1}{2}\right)\log x =y$, this boils down to Hankel's $\Gamma$-function representation, see e.g. \cite[Cor.~II.0.18]{Tenenbaum2015}. 
\end{proof}
\section{Proof of Theorem~\ref{thm:bias}}
\subsection{Character sum estimates}
\begin{proposition}\label{prop:charsums}
Let $\chi$ be a nonprincipal Dirichlet character modulo $q$. Assume GRH for the characters in \eqref{eq:fourchars}. If $\chi^2 \neq \chi_0$ we have
\[ \frac{1}{X} \int_{X}^{2X}\left|\sum_{\substack{n \le x\\ n \in S}} \chi(n)\right|^2 dx \ll  \frac{X}{(\log X)^3},\]
while if $\chi^2 =\chi_0$ we have
\[ \frac{1}{X} \int_{X}^{2X}\left|\sum_{\substack{n \le x\\ n \in S}} \chi(n) - C_q \frac{\sqrt{x}}{(\log x)^{3/4}} \left(1-\frac{\chi(2)}{\sqrt{2}}\right)^{-1/2}\sqrt{L\left(\frac{1}{2},\chi\right)L\left(\frac{1}{2},\chi\quadratic\right)}\right|^2 dx \ll \frac{X}{(\log X)^{5/2}},\]
where $C_q$ is defined in \eqref{eq:cqdef}. The implied constants depend only on $q$.
\end{proposition}
\begin{proof}
By Lemma~\ref{lem:perronshifted} with $T = X^{3/4} \asymp x^{3/4}$ and $c=1/10$ we have, uniformly for $x \in [X,2X]$,
\[\sum_{\substack{n \le x \\ n \in S}} \chi(n) = \sum_{j=1}^{m} \frac{1}{2\pi i}\int_{\mathcal{H}_{1/2+i\gamma_j}} F(s,\chi)x^s \frac{ds}{s}+ O\left( x^{\frac{1}{2}-\frac{1}{100}}\right)
\]
for any nonprincipal $\chi$. The function $F$ is defined in the first line of \S\ref{sec:prep} and is analytic in the set \eqref{eq:analyticreg} or in the set \eqref{eq:analyticreg2}, depending on $\chi$. The $m=m_{\chi}$ contours $\mathcal{H}_{1/2+i\gamma_j}=\mathcal{H}_{1/2+i\gamma_j,\chi}$ are defined in \S\ref{sec:contour}. They are Hankel loop contours going anticlockwise around zeros $1/2+i\gamma_j$ of $L(s,\chi)L(s,\chi\quadratic)$ up to height $T$ (exclusive), as well as around $s=1/2$ in case $\chi^2$ or $\chi^2 \quadratic$ is principal.
Let us write 
\[\sum_{\substack{n \le x \\ n \in S}} \chi(n) = S_1(x) + S_2(x) + O\left(x^{\frac{1}{2}-\frac{1}{100}}\right)\]
where 
$S_1(x)$ is the contribution of Hankel loops not going around $1/2$:
\[ S_1(x) =  \sum_{\substack{1 \le j \le m\\ \gamma_j \neq 0}} \frac{1}{2\pi i} \int_{\mathcal{H}_{1/2+i\gamma_j}} F(s,\chi) x^s \frac{ds}{s},\]
and $S_2$ is the contribution of the loop around $s=1/2$, in case such a loop exists:
\[ S_2(x) = \begin{cases}\frac{1}{2\pi i} \int_{\mathcal{H}_{1/2}} F(s,\chi) x^s \frac{ds}{s}& \text{if }L(1/2,\chi)L(1/2,\chi\quadratic)=0 \text{ or } \chi_0 \in \{\chi^2\quadratic, \chi^2\},\\0 & \text{otherwise.}\end{cases}\]
We shall take $r=o(1/\log X)$ in all the definitions of the loops. If $L(1/2,\chi)L(1/2,\chi \quadratic)=0$ or $\chi^2 \quadratic= \chi_0$ we have, by Lemma~\ref{lem:hankelhalf}, the pointwise bound
\[ S_{2}(x) \ll \frac{\sqrt{x}}{(\log x)^{5/4}}.\]
If $\chi^2=\chi_0$ and $L(1/2,\chi)L(1/2,\chi \quadratic)\neq 0$, we have by Lemma~\ref{lem:asymphalf} the following asymptotic relation:
\[ S_{2}(x)  = C_{q} \frac{\sqrt{x}}{(\log x)^{3/4}} \left(1-\frac{\chi(2)}{\sqrt{2}}\right)^{-1/2}\sqrt{L\left(\frac{1}{2},\chi\right)L\left(\frac{1}{2},\chi \quadratic\right)}  + O\left( \frac{\sqrt{x}}{(\log x)^{7/4}}\right).\]
In all cases,
\[ S_2(x) =  \mathbf{1}_{\chi^2 = \chi_0}C_q \frac{\sqrt{x}}{(\log x)^{3/4}} \left(1-\frac{\chi(2)}{\sqrt{2}}\right)^{-1/2}\sqrt{L\left(\frac{1}{2},\chi\right)L\left(\frac{1}{2},\chi\quadratic\right)} + O\left(\frac{\sqrt{x}}{(\log x)^{5/4}}\right).\]
It now suffices to show that 
$(1/X)\int_{X}^{2X} \left|S_1(x)\right|^2 dx \ll X/(\log X)^3$.
We have, by Lemma~\ref{lem:l2bound},
\begin{align}\label{eq:l2s1} \frac{1}{X} \int_{X}^{2X} \left|S_1(x)\right|^2 dx &\ll  \frac{1}{X} \int_{X}^{2X} \left| \sum_{\substack{1 \le j \le m \\ \gamma_j \neq 0}} \int_{\mathcal{H}_{1/2+i\gamma_j}} F(s,\chi) x^s \frac{ds}{s} \right|^2 \\
	&\ll \frac{X}{(\log X)^3}  \sum_{\substack{\gamma_1,\gamma_2 \neq 0:\\ L(1/2+i\gamma_j,\chi)=0\text{ or}\\L(1/2+i\gamma_j,\chi \quadratic)=0}} \frac{1}{|\gamma_1 \gamma_2|^{1-1/5}(1+|\gamma_1-\gamma_2|)}.
\end{align}
The sum over zeros converges by a standard argument, see \cite[Thm.~13.5]{Montgomery2007} where this is proved in the case of zeros of the Riemann zeta function. The only input needed is that between height $T$ and $T+1$ there are $\ll \log T$ zeros, which is true for any Dirichlet $L$-function, see \cite[Thm.~10.17]{Montgomery2007}.
\end{proof}

\subsection{Conclusion of proof}
Suppose $a,b$ satisfy $a\equiv b \equiv 1 \bmod (4,q)$ and $(a,q)=(b,q)=1$. Suppose the constant $C_{q,a,b}$ appearing in \eqref{eq:lincond} is positive. 
Consider $X \gg 1$ which will tend to $\infty$. By orthogonality of characters we write
\begin{equation}\label{eq:ortho}
	 S(x;q,a) - S(x;q,b) = \frac{1}{\phi(q)} \sum_{\chi_0 \neq \chi \bmod q} \overline{(\chi(a)-\chi(b))} \sum_{\substack{n \le x\\ n \in S}}\chi(n)
\end{equation}
for each $x \in [X,2X]$.
By Proposition~\ref{prop:charsums} and Cauchy-Schwarz, we can write
\begin{equation}\label{eq:finalexp} S(x;q,a)-S(x;q,b) = \frac{C_q}{\phi(q)} \frac{\sqrt{x}}{(\log x)^{3/4}} C_{q,a,b} + T(x)
\end{equation}
where
\[ \frac{1}{X} \int_{X}^{2X} \left| T(x)\right|^2 dx \ll \frac{X}{(\log X)^{5/2}}.\]
We see that in an $L^2$-sense, $T(x)$ is smaller (by a power of $\log x$) than the term of order $\sqrt{x}/(\log x)^{3/4}$ in \eqref{eq:finalexp}. To make this precise, we use Chebyshev's inequality:
\[ \mathbb{P}_{x \in [X,2X]}\left( |T(x)| \ge \frac{\sqrt{X}}{\Psi(X)(\log X)^{3/4}} \right) \ll \frac{\Psi^2(X)}{\log X} = o(1)\]
for any function $\Psi$ tending to $\infty$ slower than $(\log X)^{1/2}$. 
Here $x$ is a number chosen uniformly at random between $X$ and $2X$. It follows that $\mathbb{P}_{x \in [X,2X]}\left(S(x;q,a)>S(x;q,b)\right) \sim 1$ 
which finishes the proof. \qed

\section{Martin's conjecture}\label{sec:omega}
\subsection{Preparation}
Let $F_{\omega}(s,\chi) = \sum_{n \ge 1} \chi(n)\omega(n)/n^s$ for $\Re s >1$.
By Lemma~\ref{lem:effective0} with $\kappa = 1+1/\log x$ and $\omega(n)\ll \log n$,
\begin{equation}\label{eq:perronomega} \sum_{n \le x} \chi(n) \omega(n) = \frac{1}{2\pi i} \int_{1+\frac{1}{\log x}-iT}^{1+\frac{1}{\log x}+iT} F_{\omega}(s,\chi)x^s \frac{ds}{s} + O\left( \log x+ \frac{x \log^2 x}{T}\right)
\end{equation}
for all $T \gg 1$. We have $\omega = 1 * \mathbf{1}_{\text{Primes}}$. For $\Re s >1$ this identity leads to
\begin{equation}\label{eq:omega gen} F_{\omega}(s,\chi) = L(s,\chi) \left( \log L(s,\chi) - \frac{1}{2}\log L(2s,\chi^2) + G_{\omega}(s,\chi)\right)
\end{equation}
where $G_{\omega}$ may be analytically continued to $\Re s > 1/3$, and is bounded in $\Re s \ge 1/3 + \varepsilon$; see \cite[Eq.~(2.3)]{Meng2020}. If $\chi$ is nonprincipal, GRH for $\chi$ and $\chi^2$ implies that $F_{\omega}$ can be analytically continued to
	\begin{equation}\label{eq:fomega analytic}
	\left\{ s \in \CC: \Re s > \frac{1}{3}\right\} \setminus \{ \sigma+it :  L(1/2+it,\chi)=0 \text{ and }\sigma \le 1/2 \}
\end{equation}
	if $\chi^2$ is nonprincipal, and 
	\begin{equation}\label{eq:fomega analytic2}
\left\{ s \in \CC: \Re s > \frac{1}{3}\right\} \setminus \left(\{ \sigma+it :  L(1/2+it,\chi)=0 \text{ and }\sigma \le 1/2 \} \cup \{ \sigma : \sigma \le 1/2\}\right)
	\end{equation} 
	if $\chi^2$ is principal.
	Almost the same analysis applies for $F_{\Omega}(s,\chi) = \sum_{n \ge 1} \chi(n)\Omega(n)/n^s$, with the only change being the following variation on \eqref{eq:omega gen}: 
	\[ F_{\Omega}(s,\chi) = L(s,\chi) \left( \log L(s,\chi) + \frac{1}{2}\log L(2s,\chi^2) + G_{\Omega}(s,\chi)\right)\]
	where $G_{\Omega}$ may be analytically continued to $\Re s > 1/3$ and is bounded in $\Re s \ge 1/3+\varepsilon$.	This is a consequence of $\Omega = 1 * \mathbf{1}_{\text{Prime Powers}}$.
	The following lemma is essentially \cite[Lem.~3]{Meng2020}, and its proof is similar to the proof of Lemma~\ref{lem:perronshifted}.
\begin{lem}\label{lem:contourshiftomega}
Let $\chi$ be a nonprincipal character. Assume GRH holds for $\chi$ and $\chi^2$. Let $c \in (0,1/6)$ be a fixed constant. Let $T\gg 1$. We have
		\begin{multline}\label{eq:perronshiftedomega} \sum_{n \le x} \chi(n)\omega(n) = \sum_{j=1}^{m} \frac{1}{2\pi i}\int_{\mathcal{H}_{1/2+i\gamma_j}} F_{\omega}(s,\chi)x^s \frac{ds}{s} \\+ O\left( \log x + \frac{x\log^2 x}{T}+\frac{(x+\max\{x,T\}^{1/2}+x^{\frac{1}{2}-c}T^{c+1})\exp\left( \frac{A \log T}{\log \log T}\right)}{T}\right)
		\end{multline}
	where the list $\{1/2+i\gamma_j\}_{j=1}^{m}$ consists of the distinct nontrivial zeros of $L(s,\chi)$ with $-T'' \le t \le T'$ where $T',T''$ depend only on $T$ and $\chi$ are satisfy $T',T''=T+O(1)$. If $\chi^2$ is principal we include $1/2$ in the list. 	
	
	Here $\mathcal{H}_{\rho}$ is the truncated Hankel loop contour defined in \eqref{eq:hankel}, and it has radius $r$ which is chosen to be sufficiently small (in terms of $T$, $x$ and the list of $\gamma_j$s). The implied constant and $A$ depend only on $\chi$ and $c$.
	\end{lem}
\begin{proof}
The proof is similar to that of  Lemma~\ref{lem:perronshifted}, the main difference being the appearance of the factor $\log L(s,\chi)$ because of $\eqref{eq:omega gen}$. We need to be careful because $\log L(s,\chi)$ may be large even if $L(s,\chi)$ is small. We need to explain why the contribution of $\log L(s,\chi)$ may be absorbed into $\exp(A \log T/\log \log T)$. We shall show that $\log L(s,\chi) = O(\log T)$ holds on the relevant contour. Recall that $\arg L(s,\chi)$ is defined via $\log L(s,\chi) = \log |L(s,\chi)| + i\arg L(s,\chi)$. We have $\arg L(s,\chi)=O(\log(|t|+4))$ uniformly in $t$ and $\sigma \in [1/6,2]$, see \cite[Lem.~12.8]{Montgomery2007}. Hence our focus will be on bounding $\log |L(s,\chi)|$. By Lemmas~\ref{lem:upper} and \ref{lem:fe} we have $\log |L(\sigma+it,\chi)| \le C\log(|t|+4)$ for $|t| \gg 1$ and $\sigma \in [1/6,2]$, so that we have an easy upper bound on $\log |L(s,\chi)|$, and the focus is truly on lower bounding $\log |L(s,\chi)|$.

We want to shift the contour in \eqref{eq:perronomega} to $\Re s = 1/2 - c$ and avoid logarithmic singularities using Hankel loops. Before we do so, we replace the endpoints of the integral, namely $1+1/\log x \pm iT$, with $1+1/\log x + iT'$ and $1+1/\log x - iT''$, where $T',T'' = T+O(1)$ and the bound $\log |L(s,\chi)| \ge -C\log (|t|+4))$ holds uniformly on $\Im s = T'$ and $\Im s = -T''$ with $\sigma \in [1/6,2]$. Changing the endpoints does not affect the error term in \eqref{eq:perronomega} due to a simple variation on Remark~\ref{rem:modif}. The existence of such $T'$ and $T''$ is exactly the content of \cite[Thm.~13.22]{Montgomery2007}.

Lemmas~\ref{lem:upper}-\ref{lem:fe} allow us to bound both the vertical and horizontal contributions of $L(s,\chi)$ and $\log L(2s,\chi^2)$. The horizontal contribution of $\log L(s,\chi)$ is small due to the choice of $T'$ and $T''$. 
To bound the vertical contribution of $\log L(s,\chi)$ we use \cite[Ex.~1 at \S12.1.1]{Montgomery2007} which says that for $\Re s \ge 1/6$,  $\log L(s,\chi)=\sum_{\rho:\, |\gamma-t|\le 1}\log(s-\rho)+O(\log (|t|+4))$ unconditionally. Applying this with $\rho=1/2-c+it$ this is $ O_c(\log |t|)$ since all the zeros satisfying $|\gamma-t|\le 1$ are nontrivial and lie on $\Re s= 1/2$, and there are $\ll \log|t|$ zeros between height $t-1$ and $t+1$.
\end{proof}
The following lemma is implicit in \cite[pp.~110--111]{Meng2020}.
\begin{lem}\label{lem:mrho}
	Let $\chi$ be a nonprincipal character and suppose GRH holds for $\chi$ and $\chi^2$. 
	Let $\rho=1/2+i\gamma \neq 1/2$ be a nontrivial zero of $L(s,\chi)$. Let $m_{\rho,\chi}$ be the multiplicity of $\rho$ in $L(s,\chi)$. We have
	\[ \int_{\mathcal{H}_{\rho}} F_{\omega}(s,\chi)x^s \frac{ds}{s} = m_{\rho,\chi}\int_{\mathcal{H}_{\rho}} L(s,\chi) \log (s-\rho)x^s \frac{ds}{s}.\]
\end{lem}
\begin{proof}
Since 
	\[ F_{\omega}(s,\chi) = L(s,\chi) \log L(s,\chi) - \frac{1}{2}L(s,\chi)\log L(2s,\chi^2) +  L(s,\chi)G_{\omega}(s,\chi)\]
	and $L(s,\chi)\log L(2s,\chi^2)$, $L(s,\chi)G_{\omega}(s,\chi)$ are analytic in an open set containing $\mathcal{H}_{\rho}$, it follows that
	\[ \int_{\mathcal{H}_{\rho}} F_{\omega}(s,\chi)x^s \frac{ds}{s} = \int_{\mathcal{H}_{\rho}}L(s,\chi)\log L(s,\chi)x^s \frac{ds}{s}\]
	by Cauchy's Integral Theorem. We may write $\log L(s,\chi)$ as
	\[ \log L(s,\chi) = m_{\rho,\chi} \log (s-\rho) + H_{\rho}(s,\chi)\]
	for a function $H_{\rho}$\footnote{An estimate for $H_{\rho}(s,\chi)$ on $\mathcal{H}_{\rho}$ may be obtained, see \cite[Eq.~(2.15)]{Meng2020}.} which is analytic in an open set containing the loop, since $L(s,\chi)/(s-\rho)^{m_{\rho,\chi}}$ has a removable singularity at $s=\rho$. By Cauchy's Integral Theorem, $H_{\rho}(s,\chi)$ does not contribute to the Hankel contour integral, giving the conclusion.
\end{proof}
Lemmas~\ref{lem:contourshiftomega} and \ref{lem:mrho} hold as stated for $\Omega$ in place of $\omega$ as well.
We have the following lemma, a `logarithmic' analogue of  Lemma~\ref{lem:hankelabsolute}.
\begin{lem}\label{lem:omegaabsolute}
Let $\chi$ be a nonprincipal character and suppose GRH holds for $\chi$. 
	Let $\rho=1/2+i\gamma$ be a nontrivial zero of $L(s,\chi)$. Let 
\[ I_{\rho}:= \int_{H_{\rho}} |L(s,\chi)| |\log(s-\rho)| x^{\Re s} |ds|.\]
Then
\[ I_{\rho} \ll \sqrt{x}(|\gamma|+1)^{c+o(1)} \left( \frac{\log \log x}{\log^2 x}+\log (r^{-1}) r^2 x^r \right).\]
\end{lem}
\begin{proof}
We write $L(s,\chi)$ as $L(s,\chi)/(s-\rho)$ times $(s-\rho)$, and use Lemma~\ref{lem:ratio} to bound $L(s,\chi)/(s-\rho)$ by $(|\gamma|+1)^{c+o(1)}$. We now consider separately $|s-\rho|=r$ and $s=\rho+t$, $-c\le t\le -r$.
\end{proof}

The following is an $\omega$-analogue of Lemmas~\ref{lem:hankelhalf} and \ref{lem:asymphalf}.
\begin{lem}\label{lem:hankelhalf2}
	Let $\chi$ be a nonprincipal Dirichlet character. Assume GRH holds for $\chi$ and $\chi^2$.
	\begin{enumerate}
		\item If $L(1/2,\chi) = 0$ then
			\[ \int_{\mathcal{H}_{1/2}} F_{\omega}(s,\chi)\frac{x^s}{s} ds \ll 	\sqrt{x} \left( \frac{\log \log x}{\log^2 x}+\log (r^{-1}) r^2 x^r \right).\]
		\item If $\chi^2$ is principal and $L(1/2,\chi)\neq 0$ then
\begin{equation}\label{eq:omegasymp} \frac{1}{2\pi i}\int_{\mathcal{H}_{1/2}} F_{\omega}(s,\chi)\frac{x^s}{s} ds = -L\left(\frac{1}{2},\chi\right) \frac{\sqrt{x}}{\log x} + O\left(\frac{\sqrt{x}}{(\log x)^2}\right).
\end{equation}
\end{enumerate}
The implied constants depend only on $\chi$.
\end{lem}
\begin{proof}
The first part is a minor modification of the proof of Lemma~\ref{lem:omegaabsolute}. The second part is \cite[Eq.~(2.28)]{Meng2020}.
\end{proof}
Lemma~\ref{lem:hankelhalf2} holds for $\Omega$ in place of $\omega$, with the only difference being a sign change in \eqref{eq:omegasymp}.
\subsection{Proof of Theorem~\ref{thm:omega}}
We shall prove the theorem in the case of $\omega$; the proof for $\Omega$ is analogous.
Suppose $a,b$ satisfy $(a,q)=(b,q)=1$. Suppose the constant $D_{q,a,b}$ appearing in \eqref{eq:lincondomega} is positive.
Consider $X \gg 1$ which will tend to $\infty$. 
By orthogonality of characters we write
\begin{equation}\label{eq:orthoomega}
\sum_{\substack{n \le x\\ n \equiv a \bmod q}}\omega(n) -\sum_{\substack{n \le x\\ n \equiv b \bmod q}}\omega(n) = \frac{1}{\phi(q)} \sum_{\chi_0 \neq \chi \bmod q} \overline{(\chi(a)-\chi(b))} \sum_{n \le x}\chi(n)\omega(n)
\end{equation}
for each $x \in [X,2X]$.
By \eqref{eq:perronshiftedomega} with $T = X^{3/4} \asymp x^{3/4}$ and $c=1/10$ we have, uniformly for $x \in [X,2X]$,
\[\sum_{n \le x} \chi(n)\omega(n) = \sum_{j=1}^{m} \frac{1}{2\pi i}\int_{\mathcal{H}_{1/2+i\gamma_j}} F_{\omega}(s,\chi)x^s \frac{ds}{s}+ O\left( x^{\frac{1}{2}-\frac{1}{100}}\right)
\]
for any nonprincipal $\chi$. 
For any pair $\rho_1=1/2+i\gamma_1$, $\rho_2=1/2+i\gamma_2$ of nontrivial zeros of $L(s,\chi)$ different from $1/2$ we have, from Lemmas~\ref{lem:mrho} and \ref{lem:omegaabsolute},
\begin{multline}
	\frac{1}{X}\int_{X}^{2X}\int_{\mathcal{H}_{\rho_1}}  F_{\omega}(s_1,\chi)x^{s_1} \frac{ds_1}{s_1} \overline{ \int_{\mathcal{H}_{\rho_2}} F_{\omega}(s_2,\chi) x^{s_2} \frac{ds_2}{s_2}} \\ \ll  \frac{m_{\rho_1,\chi}m_{\rho_2,\chi}}{|\gamma_1\gamma_2|^{1-c+o(1)}\left(1+|\gamma_1-\gamma_2|\right)} X\left( \frac{\log \log X}{\log^2 X}+\log r^{-1} r^2 x^r \right)^2
\end{multline}
in analogy with Lemma~\ref{lem:l2bound}. We take $r=o(1/\log X)$. Since $m_{\rho} = O(\log (|\rho|+1))$ \cite[Thm.~10.17]{Montgomery2007} and 
\[\sum_{\substack{\gamma_1,\gamma_2 \neq 0:\\ L(1/2+i\gamma_j,\chi)=0}}\frac{1}{|\gamma_1 \gamma_2|^{1-1/5}(1+|\gamma_1-\gamma_2|)}\]
converges \cite[Thm.~13.5]{Montgomery2007}, it follows that in an $L^2$-sense, the contribution of $\rho\neq 1/2$ to \eqref{eq:orthoomega} is $O(\sqrt{x}\log \log x/(\log x)^2)$; this step corresponds to \eqref{eq:l2s1}. By Lemma~\ref{lem:hankelhalf2}, the contribution of loops around $s=1/2$ is
\[ -\frac{\sqrt{x}}{\log x}\frac{1}{\phi(q)} \sum_{\substack{\chi \bmod q\\ \chi^2 = \chi_0}} \overline{\chi(a)-\chi(b)}  L\left( \frac{1}{2},\chi\right) + O\left( \frac{\sqrt{x} \log \log x}{\log ^2 x}\right) = -\frac{\sqrt{x}}{\log x} \left( \frac{D_{q,a,b}}{\phi(q)}+o(1)\right). \]
As in the proof of Theorem~\ref{thm:bias}, Chebyshev's inequality allows us to conclude the following. The probability that for a number $x$ chosen uniformly at random from $[X,2X]$, $\sum_{n \le x, \, n \equiv a \bmod q} \omega(n) < \sum_{n \le x, \, n \equiv b \bmod q} \omega(n)$ tends to $1$ with $X$. This finishes the proof. \qed

		\section*{Acknowledgements}
We thank James Maynard and Zeev Rudnick for helpful discussions, Dor Elboim for teaching me about the Hankel contour, Dan Carmon and Max Xu for typographical corrections, Tom Bloom for pointing out the works of Kaczorowski and Andrew Granville for useful comments that substantially improved the presentation. We are thankful for the referee's thorough reading of the manuscript and valuable comments.
This project has received funding from the European Research Council (ERC) under the European Union's Horizon 2020 research and innovation programme (grant agreement No 851318).

  	\bibliographystyle{alpha}
	\bibliography{references}
	\Addresses
\end{document}